\documentclass[11pt]{elsart1p}

\let\theoremstyle\nothing

\usepackage{amsthm, amsmath, amssymb, pstricks}
\usepackage{graphicx}

\newtheorem{theorem}{Theorem}[section]
\newtheorem{lemma}[theorem]{Lemma}
\newtheorem{corollary}[theorem]{Corollary}
\theoremstyle{definition}
\newtheorem{definition}[theorem]{Definition}
\newtheorem{example}[theorem]{Example}

\newtheorem{proposition}[theorem]{Proposition}
\theoremstyle{remark}

\begin{document}

\begin{frontmatter}




\begin{keyword}
 Eulerian poset, binomial poset, Sheffer poset, triangular poset
\end{keyword}

\title{    Finite Eulerian posets which are binomial,
  Sheffer or triangular}

\author{Hoda Bidkhori}
\address{     Department of Mathematics, Massachusetts Institute of
Technology, Cambridge, Massachusetts $02139$} \ead{bidkhori@mit.edu}

\begin{abstract}

In this paper we study finite Eulerian posets which are binomial,
Sheffer or triangular. These important classes of posets are related
to the theory of generating functions and to geometry. The results
of this paper are organized as follows:
\begin{enumerate}
\item[$\bullet$]  We completely determine the structure of Eulerian
binomial posets and, as a conclusion, we are able to classify
factorial functions of Eulerian binomial posets;
\item[$\bullet$]  We give an almost complete classification of factorial functions of Eulerian
Sheffer posets by dividing the original question into several cases;
\item[$\bullet$]  In most cases above, we completely determine the
structure of Eulerian Sheffer posets, a result stronger
 than just classifying factorial functions of these Eulerian
Sheffer posets.
\end{enumerate}
We also study  Eulerian triangular posets. This paper answers
questions asked by R. Ehrenborg and M. Readdy. This research is also
motivated by the work of R. Stanley about recognizing the {\em
boolean lattice} by looking at smaller intervals.

\end{abstract}

\end{frontmatter}


\section{Introduction} \label{introduction}

The theory of binomial posets was developed in~\cite{DRS1} by
Doubilet, Rota and Stanley to formalize certain aspects of the
theory of generating functions. Binomial posets can be used to unify
various aspects of enumerative combinatorics and generating
functions. These posets are highly regular posets since the
essential requirement is that every two intervals of the same length
have the same number of maximal chains. Ehrenborg and Readdy
in~\cite{MR2} and independently Reiner in~\cite{R1} generalized the
notion of a binomial poset to a larger class of posets, which we
call Sheffer posets.

Ehrenborg and Readdy~\cite{MR1} gave a complete classification of
the factorial functions of infinite Eulerian binomial posets and
infinite Eulerian Sheffer posets, where infinite posets are posets
which contain an infinite chain. They introduced the open question
of characterizing the finite case. This paper deals with these
questions.

\

A \emph{triangular poset} is a graded poset such that the number of
maximal chains in each interval $[x,y]$ depends only on $\rho(x)$
and $\rho(y)$, where $\rho(x)$ and $\rho(y)$ are ranks of the
elements $x$ and $y$, respectively. Here we define Sheffer posets
which are special class of triangular posets. A \emph{Sheffer poset}
is a graded poset such that the number of maximal chains $D(n)$ in
an $n$-interval $[\hat{0},y]$ depends only on $\rho(y)$, the rank of
the element $y$, and the number $B(n)$ of maximal chains in an
$n$-interval $[x,y]$, where $x \neq \hat{0}$,
 depends only on $\rho(x,y)=\rho(y)-\rho(x)$. Two factorial functions
 $B(n)$ and $D(n)$ are called \emph{binomial factorial functions}
 and \emph{Sheffer factorial functions}, respectively. A  \emph{binomial poset}
is a graded poset such that the number of maximal chains $B(n)$ in
an $n$-interval  $[x,y]$ depends only on
$\rho(x,y)=\rho(y)-\rho(x)$.



\

 A graded poset $P$ is \emph{Eulerian} if every non-singleton interval of $P$
 satisfies the \emph{Euler-Poincar\'e} relation: the number of elements
 of even rank is equal to the number of elements of odd rank in
 that interval. In other words,  for all $x\leq y$ in $P$,  the
 M\"obius function is given by $\mu(x,y)=(-1)^{\rho(y)-\rho(x)}$, where $\rho$ is the rank function of $P$. Eulerian posets form an important class of posets as there are many geometric
examples such as the face lattices of convex polytopes, and more
generally, the face posets of regular CW-spheres.

\

As we mentioned above, Ehrenborg and Readdy in~\cite{MR1} classify
the factorial functions of infinite Eulerian binomial posets and
infinite Eulerian Sheffer posets. Since we are concerned here with
finite posets, we drop the requirement that binomial, Sheffer
 and triangular posets have an infinite chain.  This paper
deals with the following natural questions, as suggested by
Ehrenborg and  Readdy in~\cite{MR1}.
\begin{enumerate}
\item  Which Eulerian posets are binomial? \item  Which
Eulerian posets are Sheffer?
\end{enumerate}
We also briefly look over  Eulerian triangular posets.


\

We should mention that Stanley has proved that one can recognize
\emph{boolean lattices} by looking at smaller intervals (see
\cite{G}, Lemma~8). Farley and Schmidt answer a similar question for
\emph{distributive lattices} in~\cite{FS1}. The project of studying
Eulerian binomial posets and Eulerian Sheffer posets is also
motivated by their works. In many cases we use the factorial
function of smaller intervals to characterize the whole posets.

\subsection{Our results}

All posets considered in this paper are finite. Let us first
describe the two following poset operations:

Let $Q_i$, $i=1, \dots, k$, be posets which contain a unique maximal
element $\hat{1}$ and a unique minimal element $\hat{0}$. We define
$\boxplus_{i=1\dots k} Q_i$ to be the poset which is obtained by
identifying all of the minimal elements as well as identifying all
of the maximal elements of the posets $Q_i$. We define the
$\emph{k-summation}$ of $P$, denoted ${\boxplus^k}{(P)}$, to be
$\boxplus_{i=1\dots k} P$.

\

 Let $P$ be a poset with  $\hat{0}$. We define the \emph{dual suspension} of $P$, denoted $\Sigma^{*}(P)$,
to be the poset $P$ with two new elements $a_1$ and $a_2$.
$\Sigma^{*}(P)$ has the following order relation:
$\hat{0}<_{\Sigma^{*}(P)} {a_i} <_{\Sigma^{*}(P)} y$, for all $y>
\hat{0}$ in $P$ and $i=1, 2.$

\

Let $Q$ be a poset of odd rank. If $Q$ is an Eulerian Sheffer poset
then so is $\boxplus^{k}(Q)$. Moreover, if $P$ is an Eulerian
binomial poset, then $\Sigma^{*}(P)$ is an Eulerian Sheffer poset.

\

For Eulerian binomial posets $P$ of rank $n$, we describe their
structure depending on the value of $n$ as follows:
\begin{enumerate}

\item  $n=3$.  $P = \boxplus_{i=1\dots k} P_{q_{i}}$ for some
$q_1,\ldots ,q_r$ such that $q_i\geq 2,$ where we denote by
$P_{q}$, the face lattice of a $q$-gon.

\item  $n$ is even. $P$ is either isomorphic to $B_n$, the
boolean lattice of rank $n$, or $T_n$, the butterfly poset of rank
$n$ (defined in Definition~\ref{butterfly}).

\item $n$ is odd. $P$ is either isomorphic to
${\boxplus^{\alpha}}{(B_n)}$ or ${\boxplus^{\alpha}}{(T_n)}$ for
some positive integer $\alpha.$
\end{enumerate}

\

For Eulerian Sheffer posets $P$ of rank $n$, we describe their
structure and  factorial functions depending on the value of $n$:
\begin{enumerate}

\item  $n=3$.  $P = \boxplus_{i=1\dots k} P_{q_{i}}$ for some
$q_1,\ldots ,q_r$ such that $q_i\geq 2.$ \item $n=4$. The complete
classification of factorial functions of the poset $P$ follows from
Lemma~\ref{thm4}. \item  $n$ is odd and $n\geq 4$. Then one of the
following is true:
\begin{enumerate}
\item  $B(3)=D(3)=6$. Then $P={{\boxplus^\alpha}{(B_{n})}}$ for
some $\alpha$. \item   $B(3)=6, D(3)=8.$ This case is open.
\item  $n=5$, $B(3)=6, D(3)=10.$ This case remains open.
\item   $B(3)=6, D(3)=4.$
Then $P={\boxplus^\alpha}{(\Sigma^{*}(B_{n-1}))}$ for some $\alpha.$
\item $B(3)=4$. The classification follows from Theorems $3.11$
and $3.13$ in~\cite{MR1}.

\end{enumerate}
\item $n$ is even and $n\geq 6$. Then one of the following is
true:
\begin{enumerate}
\item $B(3)=D(3)=6.$ Then $P=B_n$.

\item  $B(3)=6, D(3)=8$. The poset $P$ has the same factorial
function as the cubical lattice of rank $n$, that is,   $D(k)=
2^{k-1}(k-1)!$ and  $B(k)=k!.$

\item  $B(3)=6, D(3)=4.$ Then
$P=\Sigma^{*}({{\boxplus^\alpha}{(B_{n-1})}})$ for some $\alpha$.

\item $B(k)= 2^{k-1}$, for $1\leq k\leq 2m$, and $B(2m+1)=\alpha\cdot2^{2m}$ for some $\alpha>1$.  In this case $P$ is isomorphic to
$\Sigma^{*}{{\boxplus^\alpha}{(T_{2m+1})}}.$
\item $B(k)= 2^{k-1}$,
$1\leq k\leq 2m+1.$ The classification follows from Theorems $3.11$
and $3.13$ in~\cite{MR1}

 \end{enumerate}

\end{enumerate}

\

The paper is structured as follows. In Section~\ref{def} we cover
some basic definitions. In Section~\ref{binomial} we completely
classify the structure of Eulerian binomial posets. See
Lemma~\ref{binomial3}, Theorems~\ref{even} and~\ref{odd}. These
results, coupled with Ehrenborg and Readdy's classification in the
infinite case, complete the classification of Eulerian binomial
posets.  In section~\ref{shefferr}, we give an almost complete
classification of the factorial functions of Eulerian Sheffer
posets. In fact, in most of above cases we completely identify the
structure of the finite Eulerian Sheffer posets, a result which is
stronger than merely classifying the factorial functions.
 In Section~\ref{triangular} we review \emph{triangular posets}.
 We classify Eulerian triangular posets such that
 the factorial functions of all of their  $3$-intervals  is equal to
 $6$. Finally, in Section~\ref{remark} we provide some conclusions and
remarks.

\section{Definitions and background} \label{def} We encourage readers
to consult Chapter $3$ of~\cite{E1} for basic poset terminology. All
the posets which are considered in this paper are finite.

\

We begin by recalling that a graded interval satisfies the
\emph{Euler-Poincar\'e relation} if it has the same number of
elements of even rank as of odd rank.
\begin{definition}
 A graded poset is \emph{Eulerian} if every non-singleton interval
satisfies the Euler-Poincar\'e relation. Equivalently, a poset $P$
is Eulerian if its M\"obius function satisfies
$\mu(x,y)={(-1)}^{\rho(x)-\rho(y)}$ for all $x\leq y$ in $P$, where
$\rho$ denotes the rank function of $P.$
\end{definition}

\

\begin{definition}
A finite poset $P$ with unique minimal element $\hat{0}$ and unique
maximal element $\hat{1}$ is called a \emph{(finite) binomial poset}
if it satisfies the following two conditions:
 \begin{enumerate}
\item Every interval $[x,y]$ is graded; in particular $P$ has
rank function $\rho$. If $\rho(x,y)=n$, then we call $[x,y]$ an
$n$-\emph{interval}. \item For all $n \in \mathbb{N}$, $n \leq $
rank$(P)$, any two $n$-intervals  have the same number $B(n)$ of
maximal chains. We call $B(n)$ the \emph{factorial function} or
\emph{binomial factorial function} of the poset $P$.
\end{enumerate} \end{definition}

\

Next, we define the \emph{atom function} $A(n)$ to be the number of
coatoms in a binomial interval of length $n$. Therefore,
$A(n)=\frac{B(n)}{B(n-1)}$ and $B(n)=A(n)\cdots A(1)$.

Consider a  binomial poset $P$. The number of maximal chains passing
through each element of rank $k$ in any interval of rank $n$ is
$B(k)B(n-k)$, for $1\leq k\leq n.$ The total number of chains in
this interval is $B(n)$. Hence, the number of elements of rank $k$
in any interval of rank $n$ is equal to
\begin{equation}\label{rank}
\frac{B(n)}{B(k) B(n-k)}.
\end{equation}

Sheffer posets were defined by Ehrenborg and Readdy~\cite{MR2}  and
independently defined by Reiner~\cite{R1}.

\

\begin{definition}
A finite poset $P$ with a unique minimal element $\hat{0}$ and a
unique maximal element $\hat{1}$ is called a \emph{(finite) Sheffer
poset} if it satisfies the following three conditions:
 \begin{enumerate}

\item  Every interval $[x,y]$ is graded; in particular, $P$ has
a rank function $\rho$. If $\rho(x,y)=n$, then we call $[x,y]$ an
\emph{$n$-interval}. \item Two $n$-intervals $[\hat{0},y]$ and
$[\hat{0},v]$ have the same number $D(n)$ of maximal chains.
\item  Two $n$-interval $[x,y]$ and $[u,v]$ such that $x\neq
\hat{0}$ and $u\neq \hat{0}$ have the same number $B(n)$ of maximal
chains.

\end{enumerate} \end{definition}

\

Let us consider a Sheffer poset $P$. An interval $[\hat{0},y]$,
where $y\neq \hat{0}$, is called a \emph{Sheffer interval} whereas
an interval $[x,y]$ with $x\neq \hat{0}$ is called a \emph{binomial
interval}. $B(n)$ and $D(n)$ are called the \emph{binomial factorial
function} and \emph{Sheffer factorial function} of $P$,
respectively. Next we define $A(n)$ and $C(n)$ to be the number of
coatoms in a binomial interval of length $n$ and a Sheffer interval
of length $n$. $A(n)$ and $C(n)$  are called the \emph{atom
function} and \emph{coatom function} of $P$, respectively. It is not
hard to see that
 $A(n)=\frac{B(n)}{B(n-1)}$ and $B(n)=A(n)\cdots A(1)$,
as well as $C(n)=\frac{D(n)}{D(n-1)}$ and $D(n)=C(n)C(n-1)\cdots
C(1)$.

The number of elements of rank $k$ in a Sheffer interval of rank $n$
is
\begin{equation}\label{q0}
\frac{D(n)}{D(k) B(n-k)}.
\end{equation}
Moreover, for a binomial interval $[x,y]$ of rank $n$ in this
Sheffer poset, the number of elements of rank $k$ is equal to
\begin{equation}\label{q1}
\frac{B(n)}{B(k)B(n-k)}.
\end{equation}

\

The \emph{dual suspension} of a poset $P$ is defined in~\cite{MR1}
as follows.
\begin{definition}\label{dual}
Let $P$ be a poset with  $\hat{0}$. We define the \emph{dual
suspension} of $P$, denoted  $\Sigma^{*}(P)$, to be the poset $P$
with two new elements $a_1$ and $a_2$. $\Sigma^{*}(P)$ has the
following order relation: $\hat{0}<_{\Sigma^{*}(P)} {a_i}
<_{\Sigma^{*}(P)} y$, for all $y> \hat{0}$ in $P$ and $i=1, 2$. That
is, the elements $a_1$ and $a_2$ are inserted between $\hat{0}$ and
atoms of $P$. Clearly if $P$ is Eulerian then so is
${\Sigma^{*}(P)}$. Moreover, if $P$ is a binomial poset then
${\Sigma^{*}(P)}$ is a Sheffer poset with the factorial function
$D_{\Sigma^{*}(P)}(n)=2B(n-1)$, for $n\geq 2.$
\end{definition}

\

\

\begin{definition}\label{sus}
Let $P$ be a poset with  $\hat{1}$. We define the \emph{ suspension}
of $P$, denoted by $\Sigma(P)$, to be the poset $P$ with two new
elements $a_1$ and $a_2$. $\Sigma(P)$ has the following order
relation:

$\hat{1}>_{\Sigma(P)} {a_i}>_{\Sigma(P)} y$, for all $y< \hat{1}$ in
$P$ and $i=1, 2$.
\end{definition}

\

\begin{definition}\label{sum}
Let $P$ be a poset with $\hat{0}$ and $\hat{1}$, and let $k$ be a
positive integer. We define the $\emph{k-summation}$ of  $P$,
 denoted  ${\boxplus^k}{(P)}$, to be the poset which is obtained by identifying all minimal elements and all maximal elements of $k$ copies of $P.$

\
\end{definition}

The \emph{dual} of poset $P$, denoted $P^{*}$, is defined as
follows: $P^{*}$ has the same set of elements as $P$ and the
following order relation, $x<_{P^{*}}y$ if and only if $y<_{P}x$.

\

\begin{definition}\label{boolean}
The \emph{boolean lattice} $B_n$ of rank $n$ is the poset of subsets
of $[n]=\{1,\cdots,n\}$ ordered by inclusion.
\end{definition}

\

\begin{definition}\label{butterfly}
The \emph{ butterfly poset} $T_n$ of rank $n$ consists of the
elements of   ${\hat{0}}\cup
({{D_{n-1}}\times{\{1,2\}}})\cup{\hat{1}}$, where
${{D_{n-1}}\times{\{1,2\}}}$ is direct product of the chain of
length $n-1$, denoted by $D_{n-1}$, and the anti-chain of rank $2$,
with the order relation $(k,i)\prec (k+1,j)$ for all $i,j\in
\{1,2\}$. Also $\hat{0}$ and $\hat{1}$ are  the unique minimal and
the unique maximal elements of this poset, respectively. Clearly,
$T_n ={\Sigma^{*}(T_{n-1})}.$
\end{definition}

\

A larger class of posets to consider is the class of
\emph{triangular posets}.

\

\begin{definition}
A finite poset $P$ with $\hat{0}$ and $\hat{1}$ is called a
\emph{(finite) triangular poset} if it satisfies the following two
conditions.
 \begin{enumerate}
\item  Every interval $[x,y]$ is graded; hence $P$ has a rank
function $\rho$. \item  Every two intervals $[x,y]$ and $[u,v]$ such
that $\rho(x)=\rho(u)=m$ and $\rho(y)=\rho(v)=n$ have the same
number $B(m,n)$ of maximal chains.
\end{enumerate} \end{definition}

\

All posets considered in this paper are finite. By binomial, Sheffer
and triangular posets, we mean finite binomial, finite Sheffer and
finite triangular posets.

\section{Finite Eulerian binomial posets} \label{binomial}

For undefined poset terminology and further information about
binomial posets, see~\cite{E1}. In this section for an Eulerian
binomial poset $P$ of rank $n$ we describe its structure as follows.

\

\begin{enumerate}
\item  If $n=3$, then  $P = \boxplus_{i=1\dots k} P_{q_{i}}$ for some $q_1,\dots ,q_r$ such that $q_i\geq 2$,  where $P_{q_{i}}$ is the face
lattice of $q_{i}$-gon. 

\item  If $n$ is an even integer, then $P= B_n$ or $T_n$. 
\item  If $n$ is an odd integer and $n\geq 5$, then there is an integer $k \geq 1$, such that  $P={\boxplus^k}{(B_n)}$ or $P={\boxplus^k}{(T_n)}$ (see Definition \ref{sum}).
\end{enumerate}





\

First we provide some examples of finite binomial posets.

\

\begin{example}The boolean lattice $B_n$ of rank $n$ is an Eulerian binomial poset with factorial function $B(k)=k!$ and atom function $A(k)=k$, $k\leq n$.
Every interval of length~$k$ of this poset is isomorphic to $B_k$.
\end{example}
\
\begin{example} Let $D_n$ be the chain containing $n+1$ elements. This poset has factorial function $B(k)=1$ and atom function $A(k)=1$, for each $k\leq n$.
\end{example}
\
\begin{example} The butterfly poset $T_n$ of rank $n$ is an Eulerian binomial poset with
 factorial function $B(k)=2^{k-1}$ for $1\leq k\leq n$ and  atom
function $A(k)=2$, for $2\leq k\leq n$, and $A(1)=1$.
\end{example}
\
\begin{example} Let ${\mathbb{F}}_q$ be the $q$-element field where $q$ is a prime power and let $V_n =V_n(q)$ be an $n$-dimensional vector space over ${\mathbb{F}}_q$. Let $L_n=L_n(q)$ denote the poset of all subspaces of $V_n$, ordered by inclusion.
 $L_n$ is a graded lattice of rank $n$. It is easy to see that every interval of size $1\leq k\leq n$ is isomorphic to $L_k$. Hence $L_n(q)$ is a binomial
 poset. This poset is not Eulerian for $q\geq 3$.
\end{example}

\

It is not hard to see that in any $n$-interval of an Eulerian
binomial poset $P$ with factorial function $B(k)$ for $1\leq k\leq
n,$ the Euler-Poincar\'e relation is stated as follows:
\begin{equation}
\sum_{k=0}^{n}{(-1)^{k}\cdot\frac{B(n)}{B(k)B(n-k)}}=0.
\end{equation}
\ The following Lemma can be found in~\cite{MR1}.

\begin{lemma}\label{gon}
Let $P$ be a graded poset of odd rank such that every proper
interval of $P$ is Eulerian. Then $P$ is an Eulerian poset.
\end{lemma}

\

\begin{lemma}\label{binomial3}Let $P$ be a Eulerian binomial poset of rank
$3$. Then the factorial function $B(n)$ for $1\leq n \leq 3$ and the
poset $P$ satisfy the following conditions:
\begin{enumerate}
\item[$(i)$] $B(2)=2$ and $B(3)=2q$, where $q$ is a positive
integer such that $q\geq 2.$ \item[$(ii)$] There is a list of
integers $q_1,\dots ,q_r$, $q_i\geq 2$, such that $P =
\boxplus_{i=1\dots k} P_{q_{i}}$, where $P_{q_{i}}$ is the face
lattice of $q_{i}$-gon.
\end{enumerate}
\end{lemma}

\begin{proof} The proof is omitted. It is a  consequence of Theorem
~\ref{rank3}.
\end{proof}

\begin{figure}[htp]
\begin{center}
\includegraphics[width=3.8in]{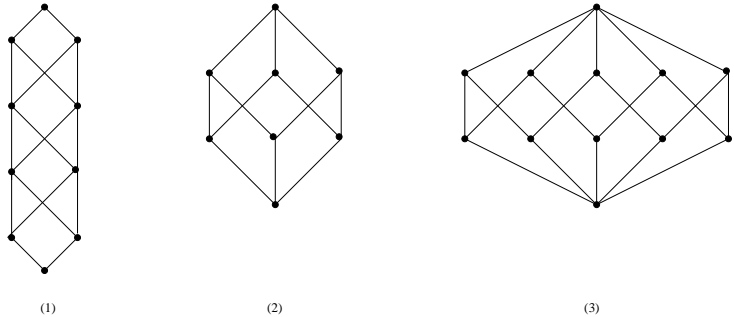}
\caption{(1): $T_5$, (2): $B_3$ and (3): $P_5$,  the face lattice of
a $5$-gon} \label{fig:1}
\end{center}
\end{figure}

\

 R. Ehrenborg and M. Readdy proved the  following two propositions in~\cite{MR1}.
 \

 \begin{proposition}\label{Tn}  Let $P$ be a  binomial poset of rank $n$ with factorial function $B(k)=2^{k-1}$ for $1\leq k\leq n$. Then the poset $P$ is isomorphic to the butterfly poset
$T_n$.
\end{proposition}
\

\begin{proposition}\label{Bn}
 Let $P$ be an Eulerian binomial poset of rank $n$ with factorial function $B(k)=k!$ for $1\leq k\leq n$. Then the poset $P$ is isomorphic to the boolean lattice $B_n$ of rank $n$.
\end{proposition}
\

It is easy to obtain the following lemma for Eulerian binomial
posets by applying the proof of Lemma $2.12$ in~\cite{MR1}.

\

\begin{lemma}\label{An}Let $P^{'}$ and $P$ be two Eulerian binomial posets of rank $2m+2$,
 $m\geq 2$, having atom functions ${A}^{'}(n)$ and $A(n)$,
respectively, which agree for $n\leq 2m$. Then the following
equality holds:
\begin{equation}\label{f4}
\frac{1}{A(2m+1)}\left(1-\frac{1}{A(2m+2)}\right)=\frac{1}{
{A}^{'}(2m+1)}\left(1-\frac{1}{{A}^{'}(2m+2)}\right).
\end{equation}
\end{lemma}

\

\begin{lemma}\label{T4} Every Eulerian binomial poset $P$ of rank $4$ is either isomorphic to $T_4$ or~$B_4.$
\end{lemma}
\begin{proof}
Applying Lemma~\ref{binomial3} gives $B(3)=2k$, where $k\geq 2$.
Eq.(\ref{q1}) implies that the number of elements of rank one is the
same as the number of elements of rank three in $P$. We denote this
number by $n.$ Hence
\begin{equation}\label{3}
n=\frac{B(4)}{B(3)B(1)}=\frac{B(4)}{B(3)}.
\end{equation}
We can also enumerate the number $r$ of elements of rank $2$ as
follows:
\begin{equation}
r=\frac{B(4)}{B(2)B(2)}.
\end{equation}
 The Euler-Poincar\'e relation on intervals of length four is $2+r=2n$.
 By enumerating the number of maximal chains, we conclude $B(4)=rB(2)B(2)=nB(3)$ and since always $B(2)=2$, we
have $2r=kn$. The Euler-Poincar\'e relation implies that $\frac{k
n}{2}+2=2n$, and so $k<4$. We have the following cases.
\begin{enumerate}
\item[$(i)$]$k=1$.  $\frac{n}{2}+2=2n$, so $n=\frac{4}{3}$. This
case is not possible. \item[$(ii)$]$k=2$.  $n+2=2n$, so $n=2$ and
$r=2$. We conclude that $B(k)=2^{k-1}$, for $1\leq k\leq 4$. By
Proposition~\ref{Tn}, $P=T_4$. \item[$(iii)$] $k=3$.
$\frac{3n}{2}+2=2n$, so $n=4$ and $r=6$. Thus $B(k)=k!$, for $1\leq
k\leq 4$. By Proposition~\ref{Bn}, $P=B_4$.
\end{enumerate}
\end{proof}

In the following theorem we obtain the structure of Eulerian
binomial posets of even rank.

\

\begin{theorem}\label{even} Every Eulerian binomial poset $P'$ of even rank $n=2m \geq 4$ is either isomorphic to $T_n$ or
$B_n$ (the butterfly poset of rank $n$ or boolean lattice of rank
$n$).
\end{theorem}
\begin{proof}
We proceed by induction on $m$. The claim is true for $2m=4$, by
Lemma~\ref{T4}. Assume that the theorem holds for
 Eulerian binomial posets of rank $2m\geq 4$. We wish to show that it also holds for Eulerian binomial posets of rank $2m+2$.

Let $P'$ be a Eulerian binomial poset of rank $2m+2$. The factorial
and atom function of this poset are denoted by $B'(n)$ and $A'(n)$,
respectively. By Lemma~\ref{T4}, every interval of size $4$ is
either isomorphic to $B_4$ or $T_4$. So the factorial function
$B'(3)$ of intervals of rank $3$, can only take the values $4,6$ and
we have the following two cases:
\

\begin{enumerate}
\item[$\bullet$]$B'(3)=6$. We wish to show that $P'$ is
isomorphic to $B_{2m+2}$ by induction on $m$. By Lemma~\ref{T4}, the
claim is true for $2m=4.$ By the induction hypothesis, the claim
holds for $n=2m$, and we wish to prove it for $n=2m+2$. Let $P=
B_{2m+2}$, so $P$ has the atom function $A(n)=n$ for $1\leq n\leq
2m+2$.
 By the induction hypothesis, $A^{'}(j)=A(j)=j$ for $j\leq 2m$. Now Lemma~\ref{An} implies that
\begin{equation}\label{f5}
   \frac{1}{A(2m+1)}\left(1-\frac{1}{A(2m+2)}\right)=\frac{1}{
A^{'}(2m+1)}\left(1-\frac{1}{A^{'}(2m+2)}\right).
\end{equation}
Since $2m=A^{'}(2m)\leq A^{'}(2m+2)<\infty$, we obtain the following
equation:
\begin{equation}
2m+1-\frac{2}{2m}< A^{'}(2m+1)< 2m+2.
\end{equation}
Thus $A^{'}(2m+1)=2m+1$. Eq.(\ref{f5}) implies that
$A^{'}(2m+2)=2m+2$. By Proposition~\ref{Bn}, the poset $P'$ is
isomorphic to $B_{2m+2}$, as desired.

\

\item[$\bullet$]$\acute{B}(3)=4$. We claim that the poset
$P^{'}$  of rank $n=2m+2$ is isomorphic to $T_n.$ By the induction
hypothesis, our claim holds for even $n\leq 2m$, and we would like
to prove it for $n=2m+2.$
 Consider the poset $T_{2m+2}$. This poset has the atom function
$A(n)=2$ for $1\leq n\leq 2m+2$. By the induction hypothesis the
intervals of length $2m$ in  $P^{'}$ are isomorphic to $T_{2m}$, so
$A^{'}(j)=2$ for $1\leq j\leq 2m$.

Clearly $2= A^{'}(2m)\leq A^{'}(2m+2)< \infty $. Eq.(\ref{f5})
implies that $2\leq A^{'}(2m+1)<4$. The case $A^{'}(2m+1)=3$ is
forbidden by similar idea that appeared in the proof of Theorem
$2.16$ in~\cite{MR1}: Assume that $A^{'}(2m+1) = 3$. Let $[x,y]$ be
a $(2m+1)$-interval in $P^{'}$. For $1 \leq k \leq 2m$ there are
$B^{\prime}(2m+1)/(B^{\prime}(k) \cdot B^{\prime}(2m+1-k))
  =
 3 \cdot 2^{2m-1} / (2^{k-1} \cdot 2^{2m-k}) = 3$
elements of rank $k$ in this interval. Let~$c$ be a coatom. The
interval $[x,c]$ has two atoms, say $a_{1}$ and $a_{2}$. Moreover,
the interval $[x,c]$ has two elements of rank $2$, say $b_{1}$ and
$b_{2}$. Moreover we know that each $b_{j}$ covers each~$a_{i}$.
Let~$a_{3}$ and~$b_{3}$ be the third atom, respectively the third
rank $2$ element, in the interval $[x,y]$. We know that~$b_{3}$
covers two atoms in $[x,y]$. One of them must be $a_{1}$ or $a_{2}$,
say $a_{1}$. But then $a_{1}$ is covered by the three elements
$b_{1}$, $b_{2}$ and $b_{3}$. But this contradicts the fact that
each atom is covered by exactly two elements. Hence this rules out
the case $A^{'}(2m+1)=3$.

 Hence $A^{'}(2m+1)=A^{'}(2m+2)=2$.
Lemma~\ref{Tn} implies that $P^{'}$ is isomorphic to $T_{2m+2}.$
\end{enumerate}

\end{proof}

\

\begin{theorem}\label{odd} Let $P$ be an Eulerian binomial poset of odd rank $n=2m+1 \geq
5$. Then $P$ satisfies one of the following conditions:
\begin{enumerate}
\item[$(i)$] There is a positive integer $k$ such that  $P$ is
the $k$-summation of the boolean lattice of rank $n$. In other
words, $P={\boxplus^k}{(B_n)}.$ \item[$(ii)$] There is a positive
integer $k$ such that $P$ is the $k$-summation of the butterfly
poset of rank $n$. In other words, $P={\boxplus^k}{(T_n)}.$

\end{enumerate}
\end{theorem}
\begin{proof}

Lemma~\ref{T4} implies that every interval of length $4$ is
isomorphic either to $B_4$ or
 $T_4$. Thus the factorial function $B(3)$ can only take the values $4$ or $6$. Therefore  we have the following two
 cases.
\

\begin{enumerate}
\item $B(3)=6$. In this case we claim that there is a positive
integer $k$ such that  $P={\boxplus^k}{(B_n)}.$ When we remove the
$\hat{1}$ and $\hat{0}$ from $P$, the remaining poset is a disjoint
union of connected components. Consider one of them and add minimal
element $\hat{0}$ and maximal element $\hat{1}$ to it. Denote the
resulting poset by $Q$. It is not hard to see that $Q$ is an
Eulerian binomial poset, and also the posets $P$ and $Q$ have the
same factorial functions and atom functions up to rank $2m$. Hence
$B_Q(k)=B_P(k)$ and $A_Q(k)=A_P(k)$, for $1\leq k\leq 2m$.
Eq.(\ref{q1}) implies that in the poset $Q$ the number of atoms and
number of coatoms are the same.  Denote this number by $t.$ Let
$x_1,\dots,x_t$ and $a_1, \dots, a_t$  be an ordering of the atoms
and coatoms of $Q$, respectively. Also, let $c_1,\dots,c_l$ be the
set of elements of rank $2m-1$ in $Q$. For each element $y$ of rank
at least $2$ in $Q$, let $S(y)$ be the set of atoms of $Q$ that are
below $y$. Set $A_i:=S(a_i)$ for each element $a_i$ of rank $2m$,
$1\leq i\leq t,$ and also set $C_i:=S(c_i)$ for each element $c_i$
of rank $2m-1$, $1\leq i\leq l.$ By considering factorial functions,
Theorem~\ref{even} implies that the intervals $[\hat{0}, a_i]$ and
$[x_j,\hat{1}]$  are isomorphic to $B_{2m}$, where $1\leq i\leq t$
and $1\leq j\leq t$. We conclude that any interval $[\hat{0},c_k]$
of rank $2m-1$ is isomorphic to $B_{2m-1}$. As a consequence,
 $|A_i|=|S(a_i)|=2m$, $1\leq i\leq t $ and also $|C_k|=|S(c_k)|=2m-1$, $1\leq k\leq l.$

In the case that there are  $i_1$ and $j_1$ such that  $A_{i_1}\cap
A_{j_1} \neq \phi,$ where $1\leq i_1,j_1\leq t$, we claim that
$2m-1\leq|A_{i_1}\cap A_{j_1}|\leq  2m$. Consider an atom $x_k \in
A_{i_1}\cap A_{j_1}$, $1\leq k\leq t$. Theorem~\ref{even} implies
that $[x_k,\hat{1}]= B_{2m}$. Thus, there is an element  $c_h$ of
rank $2m-2$ in this interval which is covered by $a_{i_1}$ and
$a_{j_1}$, $1\leq h\leq l$. Notice that $c_h$ is an element of rank
$2m-1$ in $Q$. Therefore, $|C_h|=2m-1\leq|A_{i_1}\cap A_{j_1}|\leq
|A_{i_1}|=|S(a_{i_1})|= 2m$.

We claim that for all distinct pairs $i$ and $j$, $1\leq i,j\leq t$,
we have $A_i\cap A_j \neq \emptyset.$ Associate the graph $G_Q$ to
the poset $Q$ as follows: $A_1,\dots, A_t$ are vertices of this
graph, and we connect vertices $A_i$ and $A_j$ if and only if
$A_i\cap A_j \neq \phi.$ Since  $Q-\{\hat{0},\hat{1}\}$ is
connected, we conclude that $G_Q$ is a connected graph. If
$\{A_{i_1},A _{j_1}\}$ and $\{A_{j_1},A _{k_1}\}$ are different
edges of  $G_Q$, we wish to show that $\{A_{i_1},A _{k_1}\}$ is also
an edge of $G_Q$. $|A_{i_1}\cap A _{j_1}|\geq 2m-1$ as well as
$|A_{j_1}\cap A_{k_1}|\geq 2m-1$. On other hand, since
$|A_{i_1}|=|A_{j_1}|=|A_{k_1}|=2m$, we conclude that $A_{i_1}\cap
A_{k_1} \neq \phi.$ Therefore $\{A_{i_1},A _{k_1}\}$ is also an edge
of $G_Q$. As a consequence, the connected graph $G_Q$ is a complete
graph. Thus for all different $i$ and $j$ $A_i\cap A_j \neq \phi$
and also $2m-1\leq|A_{i}\cap A_{j}|\leq 2m$, where  $1\leq i,j\leq
t$.

Now we show that $|A_i\cap A_j|=2m-1$ for different $i,j$. Suppose
this claim doesn't hold. Then there are different $i^{'},j^{'}$ such
that $|A_{i^{'}}\cap A_{j^{'}}|=2m.$  We claim that there are two
elements of rank $2m-1$ in $Q$ such that they both are covered by
coatoms $a_{i^{'}}$ and $a_{j^{'}}$. To prove this claim, consider
an atom  $x_f \in A_{i^{'}}\cap A_{j^{'}}$, so $[x_f,\hat{1}]=
B_{2m}$. Hence, there is a unique element $c_h$ of rank $2m-2$ in
this interval which is covered by both $a_{i^{'}}$ and $a_{j^{'}}$.
 By induction on $m$, Lemma~\ref{binomial3}, and the property that
$|C_h|\leq |A_{i^{'}}\cap A_{j^{'}}|=2m$ we conclude that
$[\hat{0},c_h]$ is isomorphic to $B_{2m-1}$ and so $|C_h|=2m-1$.
Therefore there is an atom  $x_d \in A_{i^{'}}\cap
A_{j^{'}}\setminus C_h.$ Since the interval $[x_d,\hat{1}]$ is
isomorphic to $B_{2m}$, there is an element $c_k\neq c_h$ of rank
$2m-1$ which is covered by coatoms $a_{i^{'}}$ and $a_{j^{'}}$ .

Since $|C_h|=|S(c_h)|=|C_k|=|S(c_k)|=2m-1$ and $C_k, C_h$ are both
subsets of $A_i\cap A_j$, we conclude that there should be an atom
$x_s \in C_k\cap C_h$. Therefore the interval $[x_s,\hat{1}]$ has
two elements  $c_k$ and $c_h$ of rank $2m-2$ such that they both are
covered by two elements  $a_i$ and $a_j$ of rank $2m-1$ in the
interval $[x_s, \hat{1}]$. We know $[x_s,\hat{1}]=B_{2m}$ and there
are no two elements of rank $2m-2$ covered by two elements of rank
$2m-1$ in $B_{2m}$. This contradicts our assumption, and so
$|A_i\cap A_j|=2m-1$ for all different $i,j$, as desired.

In summary:
\begin{enumerate}\label{states}

\item   $|A_i|=2m$  for $1\leq i\leq t$, \item  $|A_i\cap
A_j|=2m-1$ for all $1\leq i<j\leq t$, \item
$\bigcup_{i=1}^{t}A_i=\{x_1,\dots,x_t\}$.
\end{enumerate}
As a consequence, we have $t> 2m$.

Next, we are going to show that $t=2m+1.$ Without loss of
generality, consider the three different sets $A_{1}=S(a_{1}),
A_{2}=S(a_{2})$ and $A_{3}=S(a_{3})$ which are associated with the
three coatoms $a_{1}, a_{2}$ and $a_{3}.$  We know that
$|A_{1}|=|A_{2}|=|A_{3}|=2m$ and $|A_{1}\cap A_{2}|=|A_{2}\cap
A_{3}|=|A_{1}\cap A_{3}|=2m-1$. Without loss of generality, let us
that assume $A_{1}=\{x_1,x_2,\dots,x_{2m-1},y_1 \}$ and
$A_{2}=\{x_1,x_2,\dots,x_{2m-1}, y_2\}$ where $y_i\neq
x_1,\dots,x_{2m-1} $, $i=1,2$. We have two different cases: either
$A_3$ contains at least one of $y_1$ and $y_2$, or $A_3$ contains
neither of them.
 First we study the second case, $A_{3}=\{x_1,x_2,\dots,x_{2m-1},y_3\}$ where $y_3\neq y_1, y_2,
x_1,\dots, x_{2m-1}$. Considering the $t-3$ other coatoms $a_k$,
$4\leq k\leq t$, there are different atoms $y_k$, $4\leq k\leq t$,
such that $y_k\neq y_1, y_2, y_3, x_1, \dots, x_{2m-1}$ and
$A_k=S(a_k)=\{x_1,x_2,\dots,x_{2m-1}, y_k\}.$ This implies that the
number of atoms is $|\bigcup_{i=1}^{t}A_i|=t+2m-1$, which is a
contradiction.
 Hence only the first case can happen and $A_{3}$
should contain one of $y_1$ or $y_2.$ In this case $|A_{2}\cap
A_{3}|=|A_{1}\cap A_{3}|=2m-1$ implies that
$A_{3}=\{x_1,x_2,\dots,x_{2m-1},y_1,y_2\}\setminus\{x_j\}\subset
A_{1}\cup A_{2}$ for some $x_j$. Since $A_3$ was chosen arbitrarily,
it follows that for each $A_k$ we have $A_k\subset A_{1}\cup A_{2}$.
Hence
\begin{equation}
\bigcup_{i=1}^{t}A_k=\{x_1,\dots,x_{2m-1},y_1,y_2\}.
\end{equation}
Thus the number of coatoms in the poset $Q$ is $t=2m+1.$ By
Theorem~\ref{even},  $B_Q(k)=k!$, $1\leq k \leq 2m$, therefore
$B_Q(2m+1)=(2m+1)!$. By Proposition~\ref{Bn},  $Q$ is isomorphic to
$B_{2m+1}$ and so $P$ is a union of copies of $B_{2m+1}$ by
identifying their minimal elements and their maximal elements. In
other words, $P={\boxplus^k}{(B_{2m+1})}$. It can be seen that $P$
is binomial and Eulerian and the proof follows.

\

\item[$(ii)$] $B(3)=4$. With the same argument as part $(i)$, we
construct the binomial poset $Q$ by adding $\hat{1}$ and $\hat{0}$
to one of the connected components of  $P-\{\hat{0},\hat{1}\}$. We
claim that $Q$ is isomorphic to $T_{2m+1}$. Similar to part $(i)$,
let $a_1,\dots, a_t$ and $x_1,\dots,x_t$ denote  coatoms and atoms
of  $Q$. Set $A_i=S(a_i).$ By Theorem~\ref{even}, $|A_i|=2$. It is
easy to see that $\bigcup_{i=1}^{t}A_i=\{x_1,\dots,x_t\}$. Define
$G_Q$ to be the graph  with vertices $x_1,\dots,x_t$ and edges
$A_1,\dots,A_t.$ Since $Q\setminus \{\hat{0},\hat{1}\}$ is a
connected component, $G_Q$ is a connected graph. Since $[x_i,
\hat{1}]\cong T_{2m}$, the degree of each vertex of $G_Q$ is $2$ and
$G_Q$ is the cycle of length $t$. Therefore if $t>2$, $|A_i\cap
A_j|=1$ or $0$, $1\leq i<j\leq t$.

We claim $t=2.$ Suppose this claim does not hold, so $t>2$. Consider
an element $c$ of rank $3$ in $Q$, Lemma~\ref{binomial3} and
Theorem~\ref{even} imply that both intervals $[\hat{0},c]$ and
$[c,\hat{1}]$ are the butterfly posets. Hence there are two coatoms
above $c$, say $a_k$ and $a_l$, and similarly there are two atoms
below $c$, say $x_h$ and $x_s$. That is,  $A_k=A_l=\{x_h,x_k\}$. As
which is not possible when $t>2.$
 As a consequence, $t=2$ and all $A_i$'s have $2$ elements and $|\bigcup_{1}^{t}A_i|=|\{x_1,\dots,x_t\}|=2=t$.



Similar to part $(i)$,  $B_Q(k)={2^{k-1}}$ for $1\leq k\leq 2m+1$.
 By Proposition~\ref{Tn}, we conclude that $Q$ is isomorphic to
$T_{2m+1}$. Therefore, there is an integer $k>0$  such that $P={
{\boxplus}^k}{(T_{n})}.$

\end{enumerate}
\end{proof}
\begin{figure}[htp1]
\begin{center}
\includegraphics[height=1.4in]{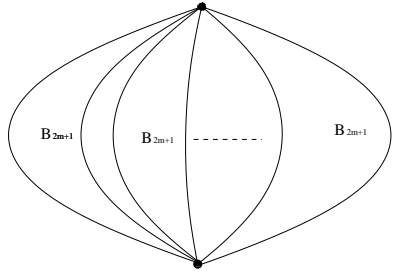}
\caption{A poset that is obtained by identifying all mimimal
elements and all maximal elements of copies of $B_{2m+1}$ }
\label{fig:2}
\end{center}
\end{figure}

\section{Finite Eulerian Sheffer posets} \label{shefferr}

For basic definitions regarding Sheffer posets, see
Section~\ref{def}.
 In this section, we give an almost complete classification of
the factorial functions and the structure of Eulerian Sheffer
posets.

First, we provide some examples of Eulerian Sheffer posets. We study
Eulerian Sheffer posets of small ranks $n=3,4$ in Lemma ~\ref{rank3}
and ~\ref{thm4}. By Lemma ~\ref{rank3} and ~\ref{thm4}, the only
possible values of $B(3)$ are $4$ and $6.$  In Section~\ref{b6},
Lemma~\ref{d} and
Theorems~\ref{evensheffer},~\ref{sheffer6},~\ref{2even},
~\ref{3even} deal with Eulerian Sheffer posets with $B(3)=6$.
Finally in Section~\ref{b4}, Theorems~\ref{4even},
~\ref{5even},~\ref{6evem} deal with Eulerian Sheffer posets with
$B(3)=4$ .

The results of this Section are summarized below.
\\

Let $P$ be a Eulerian Sheffer poset of rank $n$. Then $P$ satisfies
one of following conditions.
\begin{enumerate}

\item  $n=3$.  $P = \boxplus_{i=1\dots k} P_{q_{i}}$ for some
$q_1,\ldots ,q_r$ such that $q_i\geq 2.$ \item $n=4$. The complete
classification of factorial functions of the poset $P$ follows from
Lemma~\ref{thm4}. \item  $n$ is odd and $n\geq 4$. Then one of the
following is true:
\begin{enumerate}
\item  $B(3)=D(3)=6$. Then $P={{\boxplus^\alpha}{(B_{n})}}$ for
some $\alpha$. \item   $B(3)=6, D(3)=8.$ This case is open.
\item  $n=5$, $B(3)=6, D(3)=10.$ This case remains open.
\item   $B(3)=6, D(3)=4.$
Then $P={\boxplus^\alpha}{(\Sigma^{*}(B_{n-1}))}$ for some $\alpha.$
\item $B(3)=4$. The classification follows from Theorems $3.11$
and $3.13$ in~\cite{MR1}.

\end{enumerate}
\item $n$ is even and $n\geq 6$. Then one of the following is
true:
\begin{enumerate}
\item $B(3)=D(3)=6.$ Then, $P=B_n$.

\item  $B(3)=6, D(3)=8$. The poset $P$ has the same factorial
function as the cubical lattice of rank $n$, that is,   $D(k)=
2^{k-1}(k-1)!$ and  $B(k)=k!.$

\item  $B(3)=6, D(3)=4.$ Then
$P=\Sigma^{*}({{\boxplus^\alpha}{(B_{n-1})}})$ for some $\alpha$.

\item $B(k)= 2^{k-1}$, for $1\leq k\leq 2m$, and $B(2m+1)=\alpha
\cdot2^{2m}$ for some $\alpha>1$.  In this case $P$ is isomorphic to
$\Sigma^{*}{{\boxplus^\alpha}{(T_{2m+1})}}.$ \item $B(k)= 2^{k-1}$,
$1\leq k\leq 2m+1.$ The classification follows from Theorems $3.11$
and $3.13$ in~\cite{MR1}

 \end{enumerate}

\end{enumerate}

 It is clear that every  binomial poset is also a  Sheffer poset.
Here are some other examples of  Sheffer posets.
\\

\begin{example}
Let  $P$ be a  binomial poset of rank $n$ with the factorial
functions $B(k)$. By adjoining a new minimal element $\widehat{-1}$
to $P$, we obtain a Sheffer poset of rank $n+1$ with binomial
factorial functions $B(k)$ for $1\leq k\leq n$ and Sheffer factorial
functions, $D(k)=B(k-1)$ for $1\leq k\leq n+1$.
\end{example}

\

\begin{example}
Let $T$ be the following three element poset:

\

\begin{center}
\psset{unit=.50cm}       
\psdots*[dotscale=1](0.5,0)(1,1)(1.5,0)   
\psline(0.5,0)(1,1)(1.5,0)
\end{center}

  Let $T^{n}$  be the Cartesian product of $n$ copies of the poset $T$. The poset $C_n=T^{n}\cup \{\hat{0}\}$  is the face lattice of
an $n$-dimensional cube, also known as the \emph{cubical lattice}.
The cubical lattice is a  Sheffer poset with  $B(k)=k!$ for $1\leq
k\leq n$ and $D(k)=2^{k-1}(k-1)!$ for $1\leq k \leq n+1$.
\end{example}

\

Let $P$ be an Eulerian Sheffer poset of rank $n$. The
Euler-Poincar\'e relation for every $m$-Sheffer interval, $2\leq
m\leq n$, becomes
\begin{equation}\label{q5}
1+\sum_{k=1}^{m}{{(-1)}^{k}}\cdot\frac{D(m)}{D(k)B(m-k)}=0.
\end{equation}
It is clear that $B_2$ is the only
 Eulerian Sheffer poset of rank $2$.

In the next lemma, we characterize the structure of Eulerian Sheffer
posets of rank $3.$ The characterization of the factorial function
is an immediate consequence.

\

\begin{lemma}\label{rank3}Let $P$ be a Eulerian Sheffer poset of rank
$3$.
\begin{enumerate}
\item[$(i)$]The poset $P$ has the  factorial functions $D(2)=2$
and $D(3)=2q$, where $q$ is a positive integer such that $q\geq 2.$
\item[$(ii)$] There is a list of integers $q_1,\dots ,q_r$,
$q_i\geq 2$ such that $P = \boxplus_{i=1\dots k} P_{q_{i}}$, where
$P_{q_{i}}$ is the face lattice of a $q_{i}$-gon.
\end{enumerate}
\end{lemma}

\begin{proof}
Consider an Eulerian Sheffer poset $P$ of rank $3$. Now
$P-\{\hat{0},\hat{1}\}$ consists of elements of rank $1$ and rank
$2$ of $P$. By the Euler-Poincar\'e relation, it is easy to see that
$B(2)=2$ and every interval of length $2$ is isomorphic to $B_2$. So
in $P-\{\hat{0},\hat{1}\}$, every element of rank $2$ is connected
to two elements of rank $1$ and vice versa. Therefore, the Hasse
diagram of $P- \{{\hat{0},\hat{1}}\}$ is just the disjoint union of
the cycles of even lengths $2q_{1},\dots, 2q_{r}$ where $q_{i}\geq
2$. We conclude that $P$ is obtained by identifying all minimal
elements of the posets $P_{q_{1}},\dots, P_{q_{r}}$ and identifying
all of their maximal elements. Hence $P = \boxplus_{i=1\dots k}
P_{q_{i}}$ and $D(3)=2(q_1+\dots+ q_r)$. Thus every Eulerian Sheffer
poset of rank $3$ has the factorial functions $D(3)=2q$ where $q\geq
2$ and $B(2)=D(2)=2.$
\end{proof}

\begin{figure}[htp2]
\begin{center}
\includegraphics[height=1.1in]{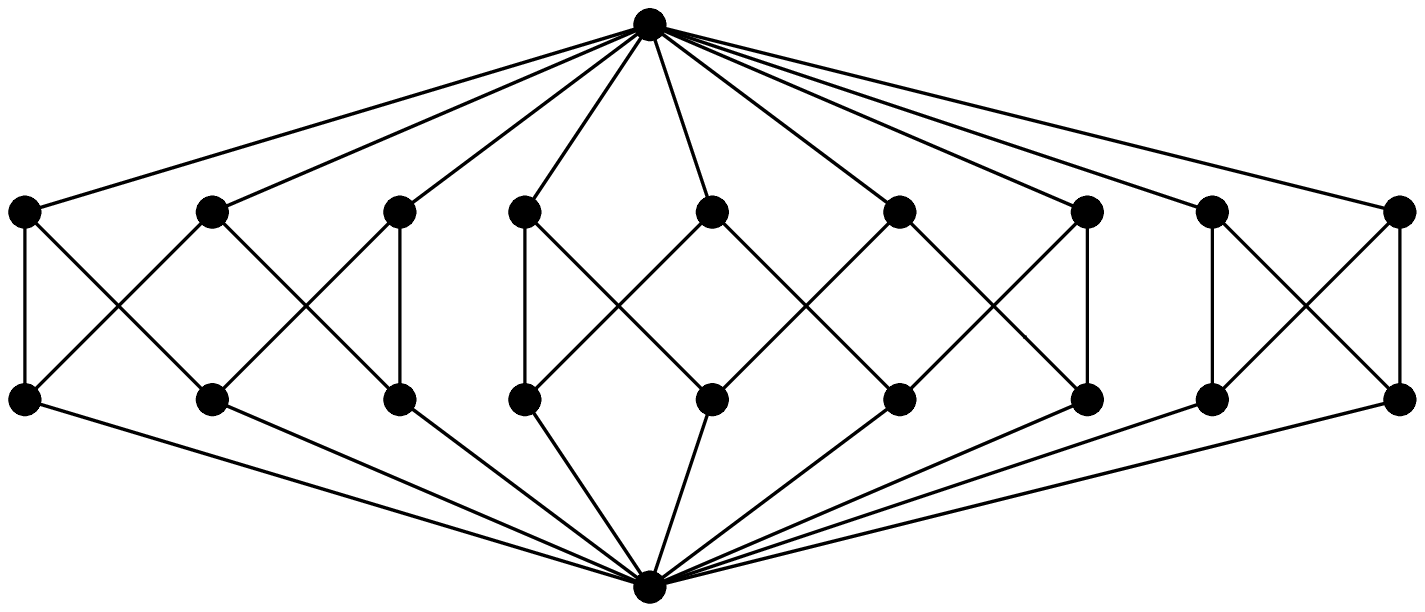}
\caption{$P = \boxplus_{k=2\dots 4}P_{k}$.}
 \label{fig:3}
\end{center}
\end{figure}


Lemma~\ref{thm4} deals with Eulerian Sheffer posets of rank $4.$

\

\begin{lemma}\label{thm4}
Let poset $P$ be an Eulerian Sheffer poset of rank $4.$ Then one of
the following conditions hold.

\begin{enumerate}
\item  $B(3)=2r$, $D(3)=4$, $D(4)=4r$, where $r\geq 2$.
\item  $B(3)=10$, $D(3)=3!$, $D(4)= 120$. \item  $B(3)=8$,
$D(3)=3!$, $D(4)=2^{3}.3!$. \item  $B(3)=3!$, $D(3)=3!$, $D(4)=4!.$
\item  $B(3)=4$, $D(3)=3!$, $D(4)=2\cdot3!$. \item
$B(3)=3!$, $D(3)=10$, $D(4)=120$. \item  $B(3)=3!$, $D(3)=8$,
$D(4)=2^{3}\cdot3!$. \item  $B(3)=3!$, $D(3)=4$, $D(4)=2\cdot3!$.
\item  $B(3)=4$, $D(3)=2r$, $D(4)=4r$ where $r\geq 2$.

\end{enumerate}
\end{lemma}
\begin{proof}
Let $P$ be an Eulerian Sheffer poset of rank $4$. Note that for
every Eulerian Sheffer poset $B(1)=D(1)=1$ as well as $B(2)=D(2)=2.$
The variables $m$, $r$, $n$ denote the number of elements of rank
$1$, $2$ and $3$ of $P$, respectively. By the Euler-Poincar\'e
relation $2+r= m+n$. The number of maximal chains in  $P$ is given
by $4r=B(3)m=D(3)n.$ Lemma~\ref{rank3} implies that there are
positive integers $k_1, k_2$ such that $D(3)=2k_2$ and $B(3)=2k_1.$
Thus $r+2=(\frac{2}{k_1}+\frac{2}{k_2})r.$ We conclude that
$\frac{2}{k_1}+\frac{2}{k_2}> 1$; therefore the case $k_1,k_2 > 3$
cannot happen. Next we study the  remaining  cases as follows.
\begin{enumerate}

\item[$(1)$] $k_2=1$. Then $n=2r$ and $2r\leq r+2$. Therefore
$r=1, 2$, and we have the following cases:
\begin{enumerate}
\item $r=1$. Then $m=1$ and $n=2$, so the Sheffer interval of
length $2$ in $P$ does not satisfy the Euler-Poincar\'e relation.
This case is not possible. \item  $r=2$. Then $n=4$ and $m=0$, which
is not possible.
\end{enumerate}

\item[$(2)$] $k_2=2$. Then $2r=2n$, so $n=r$, $m=2$ and $k_1=r$.
The fact that every interval of rank $2$ is isomorphic to $B_2$
implies that $r\geq 2$. Thus $B(1)=1$, $B(2)=2$ and $B(3)=2r$, as
well as $D(1)=1$, $D(2)=2$, $D(3)=4$, and $D(4)=4r$. The poset
$T={\Sigma^{*}(P_r)}$, where  $P_r$ is the face lattice of
$r$-polygon, is an Eulerian Sheffer poset with the described
factorial functions.

\item[$(3)$]$k_2=3$. The equation
$r+2=m+n=(\frac{2}{3}+\frac{2}{k_1})r$ implies that $k_1< 6$, so we
need to consider the following cases.
\begin{enumerate}
\item  $k_1=5$. Then $r+2=\frac{2}{5}r+ \frac{2}{3}r$, so
$\frac{1}{15}r=2$, $r=30$, $n=20$ and $m=12$. Thus $P$ has the
following factorial functions
 $B(3)=10$,  $D(3)=3!$ and $D(4)= 120$. The  face lattice of icosahedron is an Eulerian Sheffer
poset with the same factorial functions.

\item  $k_1=4$. Similarly, $P$ has the same factorial functions
as the dual of the cubical lattice of rank $4$, $B(3)=8$, $D(3)=3!$
and $D(4)=2^{3}\cdot3!$.

\item  $k_1=3$. Similarly, $P$ has the  factorial functions
$B(3)=3!$, $D(3)=3!$ and $D(4)=4!$. And $P$ is isomorphic to $B_4.$

\item $k_1=2$. Similarly, $P$ has the factorial functions
$B(3)=4$, $D(3)=3!$ and $D(4)=2\cdot3!$. The suspension of poset
$B_3$, $\Sigma(B_3)$, is an Eulerian Sheffer poset with
 the same factorial functions.
\item $k_1=1$. Then $r+2=2r+\frac{2}{3}r$, which is not
possible.
\end{enumerate}
\item[$(4)$]$k_1=3$. Then $r+2=(\frac{2}{k_1}+\frac{2}{k_2})r$
implies that $k_2<6$, so we have the following cases.

\begin{enumerate}
\item  $k_2=5$. Then $r+2=\frac{2}{5}r+ \frac{2}{3}r$, so
$\frac{1}{15}r=2$. Therefore $r=30$, $m=20$ and $n=12$ and so $P$
has the same factorial functions face lattice of a dodecahedron,
$B(3)=3!$, $D(3)=10$ and $D(4)=120$.

\item  $k_2=4$. Similarly, $P$ has the same factorial functions
as the cubical lattice of rank $4$, $B(3)=3!$ and  $D(3)=8$,
$D(4)=2^{3}\cdot3!$.

\item  $k_2=3$. $P$ has the factorial functions $B(3)=3!$,
$D(3)=3!$ and $D(4)=4!$. So, $P=B_4$.

\item $k_2=2$. It is easy to see that $P$ has the  factorial
functions as ${\Sigma}^{*}(B_3)$, $B(3)=3!$, $D(3)=2\cdot2!$ and
$D(4)=2\cdot3!$.

\item  $k_2=1$. Then $r+2=2r+\frac{2}{3}r$, which is not
possible.
\end{enumerate}

\item[$(5)$] $k_1=2$. Then $2r=2m$, so  $m=r$ and $n=2$.
Therefore , $B(3)=4$,  $D(3)=2r$ and $D(4)=4r$ where $r\geq 2$.
$T={\Sigma(P_r)}$, the suspension of poset $P_{r}$, is an Eulerian
Sheffer poset with the described factorial functions.
\item[$(7)$]$k_1=1$. Then $n=2r$ where $2r\leq r+2,$ so $r= 1,
2.$
\begin{enumerate}
\item $r=1$. Then $n=1$ and $m=2$. The Sheffer interval of
length $2$ in this poset does not satisfy the Euler-Poincar\'e
relation, so this case is not possible. \item  $r=2$. Then $m=4$ and
$n=0$, this case is not possible.

\end{enumerate}

\end{enumerate}

\end{proof}
\subsection{\bf{Characterization of the factorial functions and structure of Eulerian Sheffer posets of rank $n\geq 5$ for which $B(3)=3!.$}}\label{b6}

\
 In this section we consider Eulerian Sheffer posets of rank
$n\geq 5$ with  $B(3)=3!$. Lemma~\ref{d} shows that for any such
poset of rank $n\geq 5$, $D(3)$ can only take the values $4, 6, 8.$
In Subsections~\ref{d8},~\ref{d6},~\ref{d4}, we consider the three
different cases $D(3)=4, 6, 8$, respectively.

\

\begin{lemma}\label{d}
Let $P$ be a Eulerian Sheffer poset of rank $n\geq 6$ with
 $B(3)=3!$. Then  $D(3)$
can  take only the values $4, 6, 8.$
\end{lemma}
\begin{proof}
By Lemma~\ref{thm4}, the Sheffer factorial function of poset $P$ for
Sheffer $3$-intervals can take the following values $D(3)=4,6,8,
10$. We claim that the case $D(3)=10$ is not possible. Suppose there
is an Eulerian Sheffer poset $P$ of rank of at least $6$ with the
factorial functions $D(3)=10$ and $B(3)=3!$. By Lemma~\ref{thm4},
$P$ has the following factorial functions $D(1)=1$, $D(2)=2$,
$D(3)=10$, $D(4)=120$, $B(1)=1$, $B(2)=2!$ and $B(3)=3!$. Set
$C(6)=A$, $C(5)=B$, where $C(5)$ and $C(6)$ are coatom functions of
$P$. By Theorems~\ref{even} and~\ref{odd}, we conclude there is
$\alpha>0$ such that $B(4)=4!$ and $B(5)=\alpha.5!$. The
Euler-Poincar\'e relation implies that
$$1+\sum_{k=1}^{6}{(-1)}^{k}\cdot{\frac{D(6)}{D(k)B(6-k)}}=0,$$
therefore, by substituting the values in above equation, we have:
\begin{equation}\label{k}
2=\frac{AB}{\alpha}-AB+A,
 {\alpha}{(A-2)}={(\alpha-1)AB}.
\end{equation}
 we have the two following cases:
\begin{enumerate}
\item  $\alpha =1$. Eq.(\ref{k}) implies that $A=2$. However,
$A\geq A(5)=5$ where $A(5)$ is an atom function of $B_5$. This case
is not possible. \item  $\alpha > 1$. By Eq.(\ref{k}),
$${\left(\frac{\alpha}{\alpha-1}\right)}{\left(\frac{A}{A-2}\right)}=B.$$
$A\geq A(5)=5$ implies that $B<4$. On the other hand, since $B\geq
A(4)\geq 4$. This case is also not possible.
\end{enumerate}
We conclude that there is no Eulerian Sheffer poset of rank at least
$6$ with  $D(3)=10$ and $B(3)=3!$, as desired.
\end{proof}
\subsubsection{ \bf{Characterization of the factorial function of Eulerian Sheffer
posets of rank $n\geq 5$ for which $B(3)=3!$ and $D(3)=8.$}}
\label{d8}

\

In this subsection, we study the factorial functions of Eulerian
Sheffer posets of rank $n\geq 5$ for which $B(3)=3!$ and $D(3)=8.$
Theorem~\ref{evensheffer} characterizes the factorial functions of
such posets of even rank. However, the question of characterizing
factorial functions of Eulerian Sheffer posets of odd rank
$n=2m+1\geq 5$ with $B(3)=3!$ and $D(3)=8$ still remains open.

\
\begin{theorem}\label{evensheffer}
Let $P$ be an Eulerian Sheffer poset of even rank $n=2m+2\geq 6$
with  $B(3)=3!$ and $D(3)=8$. Then $P$ has the same factorial
functions as the cubical lattice of rank $n$, $C_n$. That is, $D(k)=
2^{k-1}(k-1)!,$ $1\leq k \leq n$ and $B(k)=k!$, $1\leq k \leq
{n-1}$.
 \end{theorem}
\

In order to prove Theorem~\ref{evensheffer}, we prove the following
three Lemmas~\ref{evens},~\ref{2m+1} and~\ref{sheffer}:

\

\begin{lemma}\label{evens}
Let $Q$ be an Eulerian Sheffer poset of odd rank $2m+1$, $m \geq 2$,
with  $B(3)=3!$. Then  $Q$ cannot have the following sequence of
coatom functions: $C(n)=2(n-1)$ for $2\leq n\leq 2m$ and
$C(2m+1)=4m+1.$
\end{lemma}
\begin{proof}
We proceed by contradiction. Assume $Q$ is such a poset.
Theorem~\ref{even} implies that $P$ has the binomial factorial
functions $B(k)=k!$ for $1\leq k\leq 2m.$ By Eq.(\ref{q0}) we
enumerate the number elements of ranks $1$, $2m-1$, $2m$ in this
Sheffer poset. Let $\{a_1,\dots,a_{4m+1}\}$, $\{e_1,\dots
,e_{(4m+1)(2m-1)}\}$ and $\{x_1,\dots, x_t\}$ denote the sets of
elements of rank $2m$, $2m-1$ and $1$ in $Q$, respectively, where
$t={\frac{4m+1}{2m}\cdot 2^{2m-1}}$. For each element $y$ of rank at
least $2$, let $S(y)$ be the set of atoms in $[\hat{0},y].$ Set
$A_{j}=S(a_{j})$ for each element $a_j$ of rank $2m$ and also
$E_{j}=S(e_{j})$ for each element $e_j$ of rank $2m-1$.
Eq.(\ref{q0}) implies that $|S(y)|=2^{r-1}$ for any element $y$ of
rank $2\leq r\leq 2m$.

We claim that for all different  $1\leq i,j \leq {4m+1}$, $A_i \cap
A_j\neq \phi$. Suppose this claim does not hold, then there exist
two different $s,l$ such that $|A_{s}\cap A_{l}|=0$ where  $1\leq s,
l \leq {4m+1}$. Since $|A_{s}|+|A_l|< t$, there is a set
$A_k=S(a_k)$ such that
 $A_k\cap({\{x_1,\dots,x_t\}-{A_{s} \cup A_{l}}})\neq \phi$, $1\leq k\leq
4m+1.$
 Generally speaking, $A_i \cap A_j$ is the set of atoms which are below $a_i\wedge a_j$.
Thus,
\begin{equation}\label{c1}
|A_i \cap A_j|=|S(a_i\wedge a_j)|=2^{rank(a_i\wedge a_j)-1}.
\end{equation}
 Let us recall the following facts:
\begin{enumerate}
\item $A_k \cap ({\{x_1,\dots,x_t\}-{A_s \cup A_l}})\neq \phi$
\item $|A_i \cap A_j|=|S(a_i\wedge a_j)|=2^{rank(a_i\wedge
a_j)-1}$ for all different $i, j$, $1\leq i, j\leq 4m+1.$
\end{enumerate}
The above equations yield $|A_l\cap A_k|, |A_s\cap A_k|\leq
2^{2m-2}.$ Furthermore, since
$|\{x_1,\dots,x_t\}|=t={\frac{4m+1}{2m}\cdot 2^{2m-1}}$,
$|A_l|=|A_s|=|A_k|=2^{2m-1}$ and $|A_l\cap A_s|=0$, we conclude that

\begin{equation}\label{q6} |A_k \cap ({\{x_1,\dots,x_t\}-{A_{s}\cup
A_{l}}})|\leq |\{x_1,\dots,x_t\}-{A_{s} \cup A_{l}}|=
{\frac{2^{2m-1}}{2m}}.
\end{equation}
Since $|A_l\cap A_k|, |A_s\cap A_k|\leq 2^{2m-2},$ Eq.(\ref{q6})
implies that $|A_{l}\cap A_k|,|A_s\cap A_k|\neq 0.$ Consider an
arbitrary atom $x_1 \in A_l\cap A_k$. Clearly $a_l$, $a_k$ are
elements of rank $2m-1$ in the interval $[x_1, \hat{1}]$. By
Proposition~\ref{Bn}, $[x_1, \hat{1}]=B_{2m}$. So $a_l\wedge a_k$ is
covered by $a_k$ and $a_l$, therefore $|A_l\cap A_k|= 2^{2m-2}$ and
similarly, $|A_s \cap A_k|= 2^{2m-2}.$

We have seen $|A_s\cap A_k|=2^{2m-2}$, $|A_l\cap A_k|=2^{2m-2}$ and
$|A_k|=2^{2m-1}$. Moreover, since we assumed $A_s \cap A_l=\phi $,
we conclude that $A_k=A_s\cup A_l$. On the other hand $A_k\cap
(\{x_1,\dots,x_t\}-{A_s\cup A_l})\neq \phi $, which is not possible
when $A_k=A_s\cup A_l$. This contradicts our assumption.
Therefore $|A_i \cap A_j|\neq 0$ for $1\leq i,j \leq {4m+1}$. So for
every distinct pair $a_i$ and  $a_j$, there is an atom $x_h \in
A_i\cap A_j$. As above $[x_h, \hat{1}]=B_{2m}$, so there is at least
one element of rank $2m-2$ in this interval, $e_k$, $1\leq k \leq
{(4m+1)(2m-1)},$ and it is covered by both $a_i$ and $a_j$. In
addition, for every element $e_l$ of rank $2m-1$ in $Q$ $[e_l,
\hat{1}]$ is isomorphic to $B_2$. As a consequence, for every $e_l$
there is exactly one pair $a_i,a_j $ such that $e_l$ is covered by
them. Hence, the number of the disjoint pairs of elements of rank
$2m$ in poset $Q$ is at most the number of elements of rank $2m-1$.
That is, ${(4m+1)(2m-1)}\geq {(4m+1)(2m)}$ which is not possible.
This contradicts the assumption. So there is no poset $Q$ with the
described factorial and coatom functions, as desired.
\end{proof}
Lemma~\ref{evens}, implies the following.
\
\begin{corollary}\label{sheffer1}
Let $P$ be an Eulerian Sheffer poset of rank $2m+2$, $m \geq 2$,
with  $B(k)=k!$, for $1\leq k\leq 2m$.  $P$ cannot have the
following sequence of coatom functions: $C(n)=2(n-1)$, $2\leq n\leq
2m$, $C(2m+1)=4m+1$ and $C(2m+2)=4(2m+1).$
\end{corollary}
\
\begin{lemma}\label{2m+1}
Let $Q$ be an Eulerian Sheffer poset of rank $2m+2$, $m \geq 2$ with
the binomial factorial functions $B(k)=k!$, for $1\leq k\leq 2m+1$.
Then $Q$ cannot have the following sequence of coatom functions:
 $C(n)=2(n-1)$ for $2\leq n\leq 2m$, $C(2m+1)=4m-2$ and $C(2m+2)=2m+1$.
\end{lemma}
\begin{proof}
We proceed by contradiction, assume that $Q$ is such a poset of rank
$2m+2$ as described above. By Eq.(\ref{q0}), we can   enumerate the
number elements of rank $k$ in this Sheffer poset of rank $n=2m+2$
for $1\leq k\leq n.$ Let $\{a_1,\dots,a_{2m+1}\}$,
$\{e_1,\dots,e_{{{(2m)}^2}-1}\}$ and $\{x_1,\dots,x_t\}$, where
$t={\frac{4m-2}{2m}\cdot2^{2m-1}}$, be the sets of elements of rank
$2m+1$, $2m$ and $1$ in poset $Q$, respectively.

With the same argument as Lemma~\ref{evens}, for any element $y$ of
at least $2$ we define  $S(y)$ to be the set of atoms in interval
$[\hat{0},y].$ Set $A_{j}=S(a_{j})$ for $1\leq j \leq {2m+1},$ and
also set $E_{j}=S(e_{j})$, $1\leq j \leq {{{(2m)}^2}-1}$. By
Eq.(\ref{q0}), $|E_{j}|=|S(e_j)|=2^{2m-1}$ for $1\leq j \leq
{{{(2m)}^2}-1}$ and also $|A_{i}|=|S(a_i)|=\frac{4m-2}{2m}\cdot
2^{2m-1}=t$, $1\leq i\leq 2m+1.$

For each element $e_i$ of rank $2m$,  $[e_i,\hat{1}]=B_2$. Hence,
each element $e_i$ of rank $2m$  covered by exactly two coatoms such
as $a_r, a_s$ where $1\leq r,s\leq 2m+1$ in $Q$. By Eq.(\ref{q0}),
the number of elements of rank $2m$ is $(2m)^2-1$ and also the
number of pairs of elements of rank $2m+1$ is $m(2m+1).$ We deduce,
there are at least two different coatoms such as $a_k, a_l$ that
both cover two different elements  $e_i, e_j$ for some particular
$i,j$ . We know the following facts:
\begin{enumerate}
\item $|A_k|=|A_l|={\frac{4m-2}{2m}\cdot2^{2m-1}}=|\{x_1, \dots,
x_t\}|=t$ \item $|E_i|=|E_j|=2^{2m-1}$ \item $E_i, E_j \subseteq A_k
=A_l =\{x_1,\dots,x_t\}$.
\end{enumerate}

By the above facts $|E_i|+|E_j|> |A_k|,|A_l|.$ Hence, there is at
least one atom $x_r \in E_i, E_j, A_k, A_l$ such that $e_i, e_j$ are
elements of rank $2m-1$ in the intervals $[x_r,a_l ]$ and
$[x_r,a_k]$ . By Proposition~\ref{Bn}, $[x_r,a_k]=[x_r,a_l]=B_{2m}$,
so there is an element $c$ of rank $2m-2$ in this interval $[x_r,a_l
]$ which is covered by $e_i$ and $e_j$. Therefore the interval $[c,
\hat{1}]$ has two elements $e_i,e_j$ of rank $1$ and they both are
covered by two elements $a_k,a_l$ of rank $2$. By
Proposition~\ref{Bn}, $[c,\hat{1}]=B_3.$ Since $B_3$, dose not have
two elements of rank $1$ which are  both covered by two elements of
rank $2$, it lead us to contradiction. There is no poset with
described conditions, as desired.


\end{proof}

\

\begin{lemma}\label{sheffer}
Let $Q$ be an Eulerian Sheffer poset of rank $2m+2$, $m \geq 2$,
with binomial factorial function $B(k)=k!$ for $1\leq k\leq 2m$.
Then the poset $Q$ cannot have the following sequence of coatom
functions: $C(n)=2(n-1)$, $2\leq n\leq 2m$, $C(2m+1)=4m-1$ and
$C(2m+2)=\frac{4}{3}(2m+1).$
\end{lemma}
\begin{proof}
We proceed by contradiction. So, suppose
 $Q$ is such a poset of rank $2m+2 $ with the described factorial
functions. We enumerate the number of elements of rank $k$ in $Q$ as
 follows,
\begin{equation}\label{f2}
{\frac{D(2m+2)}{B(k)D(2m+2-k)}=\frac{C(2m+2)\cdots
C(2m+2-k+1)}{k!}.}
\end{equation}
 Thus, $\{a_1,\dots,a_{\frac{4}{3}{(2m+1)}}\}$, $ \{e_1,\dots, e_{\frac{4}{6}{(2m+1)(4m-1)}}\}$ are the sets of elements of rank $2m+1$ and $2m$ in $Q$,
 respectively. For every element $e_i$ of rank $2m$, $[e_i,\hat{1}]$ is isomorphic
to $B_2$. So, each element of rank $2m$ covered by exactly two
different elements of rank $2m+1$.

There are exactly ${\frac{4}{6}{(2m+1)(4m-1)}}$ elements of rank
$2m$ in $Q$, and we also know that there are
$(\frac{4}{6}(2m+1))(\frac{4}{3}(2m+1)-1)$ different pairs of
coatoms $\{a_i,a_j\}$ in $Q$, $1\leq i < j \leq
{\frac{4}{3}(2m+1)}.$  We conclude there are at least two different
coatoms $a_k,a_l$ such that they both cover two different elements
$e_i, e_j$ of rank $2m$.
 The interval $T=[\hat{0},a_k]$ has
binomial factorial functions $B_T(k)=k!$ for $1\leq k\leq 2m$ and
 coatom functions $C_T(n)=2(n-1)$ for $2\leq n\leq
2m$ and $C_T(2m+1)=4m-1$. Let $\{y_1,\dots, y_t\}$ be the set of
atoms in poset $T$ where $t={\frac{(4m-1)}{2m}.2^{2m-1}}$. Thus
$A_{k}=\{y_1,\dots, y_t\}$. Set $E_{j}=S(e_j), E_{i}=S(e_i)$, so
$E_j, E_i \subset A_k$. By Eq.(\ref{q0}),
$|E_{i}|=|E_{j}|=2^{2m-1}$, therefore $|E_i|+|E_j|>|A_k|$. We
conclude that there is at least one atom $y_1\in T$ which is below
$e_i, e_j$ and $a_k$.

 Proposition~\ref{Bn}, implies that $[y_1,a_k]=B_{2m}$.
 By the boolean lattice properties, there is an
element $c$ of rank $2m-2$ in $[y_1,a_k]$ such that $c$ is covered
by $e_i,e_j$. By Proposition~\ref{Bn},  $[c,\hat{1}]=B_3$. Consider
the interval $[c,\hat{1}]$, $a_k$ and $a_l$ are two elements of rank
$2$ in this interval and they both cover two elements $e_i$ and
$e_j$ of rank $1$. It contradicts the fact that $[c,\hat{1}]=B_3$.
We conclude that $[c,\hat{1}]\neq B_3$. It lead us to contradiction,
there is no poset $Q$ with describe conditions, as desired.

\end{proof}
 The following lemma can be obtained by
applying the proof of Lemma 4.8 in~\cite{MR1}.

\

\begin{lemma}\label{MRR}
Let $P$ and $P^{'}$ be two Eulerian Sheffer posets of rank $2m+2$,
$m\geq 2$, such that their binomial factorial functions and coatom
functions  agree up to rank $n\leq 2m$. That is $B(n)=B^{'}(n)$ and
$C(n)=C^{'}(n),$ where $m\geq 2$. Then the following equation holds,
\begin{equation}\label{mrf}
\frac{1}{C(2m+1)}\left(1-\frac{1}{C(2m+2)}\right)=
\frac{1}{C^{'}(2m+1)}\left(1-\frac{1}{C^{'}(2m+2)}\right).
\end{equation}
\end{lemma}

\

\begin{proof}[Proof of Theorem~\ref{evensheffer}]
Let $C(k)$ and ${C^{'}}(k)=2(k-1)$ be the coatom functions of the
Eulerian Sheffer poset $P$ and $C_n$, the cubical lattice of rank
$n$, for $2\leq k\leq n=2m+2$.  We only need to show that
$C(n)={C^{'}}(n)=2(n-1)$ for $2\leq n\leq 2m+2$ We prove this claim
by induction on $m.$ By Lemma~\ref{thm4}, $C(4)=C^{'}(4)=6$ and the
claim is hold for $m=1.$ By induction hypothesis,
$C(n)=C^{'}(n)=2(n-1)$ for $2\leq n\leq 2m$. Set $B=C(2m+1)$ and
$A=C(2m+2)$. Theorem~\ref{odd} implies that $B(k)=k!$ for $1\leq
k\leq 2m$ and there is a positive integer $\alpha$ such that
$B(2m+1)=\alpha(2m+1)!$. We know that $D(k)=2^{k-1}.(k-1)!$ for
$1\leq k\leq 2m$, so $D(2m+1)=B2^{2m-1}(2m-1)!$ and
$D(2m+2)=AB2^{2m-1}(2m-1)!$. Since $P$ is  an Eulerian Sheffer
poset, the {Euler-Poincar\'e} relation implies that,
\begin{equation}\label{q8}
1+\sum_{k=1}^{2m+2}{\frac{(-1)^kD(2m+2)}{D(k)B(2m+2-k)}}=0.
\end{equation}
By substituting the values of the factorial functions we have,
\begin{equation}\label{q7}
2-A+\frac{AB}{2}\left[\frac{1}{2m}-\frac{1}{2m(2m+1)}+\frac{2^{2m}}{2m(2m+1)}-\frac{2^{2m}}{2\alpha
m(2m+1)}\right]=0.
\end{equation}
Thus,
\begin{equation}\label{q9}
A\left(1-B\left(\frac{2\alpha{m}+(\alpha-1)2^{2m}}{4{\alpha}m(2m+1)}\right)\right)=2.
\end{equation}
We can see that if $\alpha\geq 2$, the left side of Eq.(\ref{q9})
become negative, therefore $\alpha=1$ and posets $P$ and $C_{2m+2}$
have the same binomial factorial functions. Since $2m+1=A(2m+1)\leq
C(2m+2)< \infty$, Lemma~\ref{MRR} implies that $4m-2\leq
C(2m+1)=B\leq 4m+1$.  Since $\alpha=1,$ Eq.(\ref{q9}) implies that
$2-A+\frac{AB}{4m+2}=0$. Therefore $A$ and $B$ should satisfy one of
the following cases:
\begin{enumerate}
\item[$(1)$] $B=4m-2$ and $A=2m+1$. \item[$(2)$] $B=4m-1$
and $A=\frac{4}{3}(2m+1)$. \item[$(3)$] $B=4m$ and $A=4m+2$.
\item[$(4)$] $B=4m+1$ and $A=4(2m+1)$.
\end{enumerate}
As we have discussed in Corollary~\ref{sheffer1} as well as
Lemma's~\ref{2m+1} and ~\ref{sheffer}, the cases $(1)$, $(2)$, $(4)$
are not possible. The case $(3)$ happens in the cubical lattice of
rank $2m+2$, $C_{2m+2}$. Thus, $P$ has same factorial functions as
$C_{2m+2}$, as desired.
\end{proof}

 Classification of the factorial functions of Eulerian
Sheffer posets of odd rank $n=2m+1\geq 5$ with $B(3)=6$ and $D(3)=8$
is still remaining open. Let $\alpha$ be a positive integer and set
$Q_{\alpha}={{\boxplus^{\alpha}}{(C_{2m+1})}}$. It can be seen that
$Q_{\alpha}$ is an Eulerian Sheffer poset and it has the following
factorial functions $D(k)= 2^{k-1}(k-1)!$ for $1\leq k \leq n-1$,
$D(n)={\alpha}\cdot 2^{n-1}(n-1)! $ and $B(k)=k!$ for $1\leq k \leq
{n-1}.$ We ask the following question:

\

\textbf{Question}: Let $P$ be an Eulerian Sheffer poset of odd rank
$n=2m+1\geq 5$ with  $B(3)=6, D(3)=8$. Is there  a positive integer
$\alpha$, such that  $P$ has the same factorial function as poset
$Q_{\alpha}={{\boxplus^{\alpha}}{(C_{2m+1})}}$, where $C_{2m+1}$ is
a cubical lattice of rank $2m+1$.

 \subsubsection{\bf{Characterization of the structure of Eulerian Sheffer posets of rank $n\geq 5$ for which $B(3)=3!$,
 and $D(3)=3!=6.$}}\label{d6}

\

In this section, we prove the following: \

\begin{theorem}\label{sheffer6} Let $P$ be an Eulerian Sheffer poset of rank $n \geq
3$ with  $B(3)=D(3)=3!=6$ for $3$-intervals. $P$ satisfy one of the
following cases:
\begin{enumerate}
\item[$(i)$]  $n$ is an odd.  There is an integer $k \geq 1$
such that $P={{\boxplus^k}{(B_{n})}}$. \item[$(ii)$] $n$ is an even.
$P=B_n$.
\end{enumerate}
\end{theorem}
\begin{proof}
We proceed by induction on $n$. Theorem ~\ref{rank3} and
Lemma~\ref{thm4} imply that this theorem holds for $n=3,4$. Assume
that Theorem~\ref{sheffer6} holds for $n\leq m$, we wish to show
that it also holds for $n=m+1\geq 5$. This problem divides into the
following  cases:
\begin{enumerate}
\item[$(i)$] $n=m+1$ is odd. Consider poset $Q$ which  is
obtained by adding $\hat{0}$ and $\hat{1}$ to a connected component
of $P-\{\hat{0},~\hat{1}\}$. So
 $Q$ is an Eulerian Sheffer poset with  $B(3)=D(3)=3!=6$ . By induction
hypothesis, every intervals of rank $k\leq m$ isomorphic to $B_k$.
So, the Sheffer and binomial factorial functions  of $Q$ the boolean
lattice of rank $m+1$ are the same up to rank $m=n-1$. Therefore,
$Q$ and also $P$ are binomial posets. Theorem~\ref{odd} implies
there is a positive integer $k$ such that
$P={{\boxplus^k}{(B_{n})}}$, as desired.
\

\item[$(ii)$] $n=m+1$ is even. We proceed by induction on  $n$, the
rank of $P$. Let $C(k)$ and $C^{'}(k)=k$ be the coatom functions of
posets $P$ and $B_n$, respectively, where $k\leq n.$ By induction
hypothesis $C(k)=C^{'}(k)$ for $k\leq n-2.$ So, Lemma~\ref{MRR}
implies that
\begin{equation}\label{w1}
\frac{1}{C(n-1)}\left(1-\frac{1}{C(n)}\right)=
\frac{1}{C^{'}(n-1)}\left(1-\frac{1}{C^{'}(n)}\right).
\end{equation}
By induction hypothesis,  there is a positive integer $\alpha$ such
that $C(n-1)=\alpha(n-1)$. Moreover, we know that $C^{'}(n-1)=n-1$
and $C^{'}(n)=n$. Eq.(\ref{w1}) implies that $\alpha=1$ and
$C(n)=n$, so poset $P$ has the same factorial function as $B_n$ and
$P=B_n$, as desired.
\end{enumerate}
\end{proof}

\subsubsection{\bf{Characterization of the structure of Eulerian Sheffer posets of
rank $n\geq 5$ for which $B(3)=3!$ and $D(3)=4$ }}\label{d4}

\

Let $P$ be an Eulerian Sheffer poset of rank $n\geq 5,$ with
$B(3)=3!$ and $D(3)=4$. In this section we show that in case
$n=2m+2$, $P=\Sigma^{*}({{{\boxplus^\alpha}{(B_{2m+1})}}})$ for some
$\alpha \geq 1$ and in case $n=2m+1$,
$P={{\boxplus^\alpha}{(\Sigma^{*}(B_{2m}))}}$, for some $\alpha \geq
1$.

\begin{figure}[htp3]
\begin{center}
\includegraphics[height=1.8in]{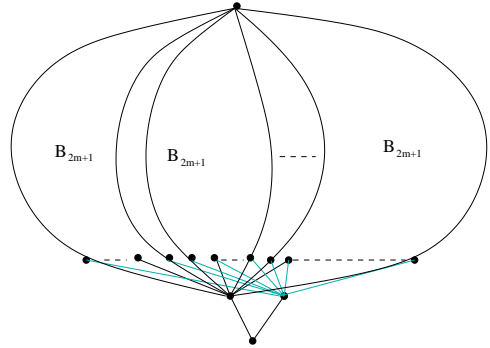}
\caption{$\Sigma^{*}({{{\boxplus^\alpha}{(B_{2m+1})}}})$}
\label{fig:4}
\end{center}
\end{figure}

 \begin{theorem}\label{2even}
Let $P$ be an Eulerian Sheffer poset of even rank $n=2m+2\geq 4$
with  $B(3)=3!$ and $ D(3)=4$.
Then $P=\Sigma^{*}({{{\boxplus^\alpha}{(B_{2m+1})}}})$, where
$\alpha=\frac{B(2m+1)}{(2m+1)!}$ is positive integer, as a
consequence $P$ has the following binomial and Sheffer factorial
functions:
\begin{enumerate}
\item[] $B(k)=k!$ for  $1\leq k\leq 2m$, and
$B(2m+1)=\alpha(2m+1)!,$ \item[] $D(1)=1$, $D(k)=2(k-1)!$ for $2\leq
k \leq 2m+1$, and $D(2m+2)=2\alpha(2m+1)!.$
\end{enumerate}
\end{theorem}
\begin{proof}
By Theorem~\ref{odd}, we know that there is a positive integer
$\alpha$ such that $P$ has the binomial factorial function
$B(2m+1)=\alpha (2m+1)!$ and $B(k)=k!$, $1\leq k<n=2m+1$.
 We proceed by induction on $m.$ The case $m=1$ implies that
$\alpha =1$, $B(3)=3!$ and $D(3)=4.$ By applying Lemma ~\ref{thm4},
it can be seen that the poset $P$ has the same  factorial functions
as $\Sigma^{*}{B_3}$; therefore, poset $P$ has two atoms and its
binomial $3$-intervals are isomorphic to $B_3.$ We conclude that
$P=\Sigma^{*}{B_3}$  and so Theorem ~\ref{2even} holds for $m=1.$
 In case $m>1$, by Theorem~\ref{odd}, $Q={{\boxplus^\alpha}{(B_{2m+1})}}$ is the
only Eulerian binomial poset of rank $2m+1$ with the binomial
factorial functions $B(k)=k!$ for $1\leq k \leq 2m$ and
$B(2m+1)=\alpha (2m+1)!,$  where $\alpha$ is a positive integer. Set
$P^{'}=\Sigma^{*}{Q}=\Sigma^{*}{{{\boxplus^\alpha}{(B_{2m+1})}}}$.
It can be seen that $P^{'}$ is an Eulerian Sheffer poset of rank
$2m+2$ with  coatom functions $C^{'}(2m+2)= \alpha(2m+1)$ and
$C^{'}(k)=(k-1)$ for $3\leq k \leq {2m+1} $ as well as $C^{'}(2)=2.$

By induction hypothesis, the theorem holds for $m-1$ and $n=2m$. We
wish to show it also holds for $m$ and $n=2m+2.$ Let $C(k)$, $2\leq
k\leq 2m+2$, be the coatom function of $P$ of rank $2m+2$ which
satisfies Theorem conditions. By induction hypothesis, $C(k)=2(k-1)$
for $2\leq k\leq 2m$. Lemma~\ref{MRR}, implies the following
Eq.(\ref{g1})
\begin{equation}\label{g1}
\frac{1}{C(2m+1)}\left(1-\frac{1}{C(2m+2)}\right)=
\frac{1}{C^{'}(2m+1)}\left(1-\frac{1}{C^{'}(2m+2)}\right).
\end{equation}
 By substituting the values of $C^{'}(2m+2)$ and $C^{'}(2m+1)$, we
 have
\begin{equation}\label{g2}
\frac{1}{C(2m+1)}\left(1-\frac{1}{C(2m+2)}\right)=
\frac{1}{2m}\left(1-\frac{1}{\alpha(2m+1)}\right).
\end{equation}
The poset  $P$ has the binomial factorial functions $B(2m+1)=\alpha
(2m+1)!$, where $\alpha$ is a positive integer, and $B(k)=k!$ for
$1\leq k<2m+1$. We conclude that $A(2m+1)=\alpha(2m+1)$ and
$A(2m)=2m$. So $C(2m+2)\geq A(2m+1)=\alpha(2m+1)$ as well as
$C(2m+1)\geq A(2m)=2m.$ Eq.(\ref{g2}) implies that $C(2m+1)=2m$ and
also $C(2m+2)={\alpha(2m+1)}.$ By induction hypothesis,
$D(k)=2(k-1)!$ for $2\leq k \leq 2m.$ Since $C(2m+1)=2m$ as well as
$C(2m+2)={\alpha(2m+1)},$ we conclude that
 $P$ has the same factorial functions as poset
$P^{'}={\Sigma}^{*}({{{\boxplus^\alpha}{(B_{2m+1})}}}).$

Applying Eq.(\ref{q0}), $P$ has $\frac{D(2m+2)}{B(2m+1)}=2$ elements
of rank $1$, let us call them ${\hat{0}}_1$ and ${\hat{0}}_2$. Using
Eq.(\ref{q0}), the number elements of rank $1\leq k\leq 2m+1$ in
posets $[{\hat{0}}_1, \hat{1}]$ and $[{\hat{0}}_2, \hat{1}]$ is

\begin{equation}\label{alpha}
\frac{{\alpha}{(2m+1)!}}{k!(2m+1-k)!}.
\end{equation}
The intervals $[{\hat{0}}_1,\hat{1}]$ and $[{\hat{0}}_2,\hat{1}]$
both have the factorial functions, $B(k)=k!$ for $1\leq k\leq 2m$
and $B(2m+1)=\alpha(2m+1)!. $ It can be seen that the intervals
$[{\hat{0}}_1,\hat{1}]$ and $[{\hat{0}}_2,\hat{1}]$ satisfy the
Euler-Poincar\'e relation and these intervals are Eulerian and
binomial. Applying Theorem~\ref{odd} implies that both intervals
$[{\hat{0}}_1,\hat{1}]$ and $[{\hat{0}}_2,\hat{1}]$ are isomorphic
to the poset ${{\boxplus^\alpha}{(B_{2m+1})}}$. Since, $P$ has the
same factorial functions as poset
$P^{'}={\Sigma}^{*}({{{\boxplus^\alpha}{(B_{2m+1})}}})$,
Eq.(\ref{q0}) yields that the number of elements of rank $k+1$ in
$P$ is the same as the number of elements of rank $k$ in intervals
$[{\hat{0}}_1,\hat{1}]$ and $[{\hat{0}}_2,\hat{1}]$ for $1\leq k\leq
2m+1$, that is
\begin{equation}\label{eq}
\frac{\alpha{(2m+1)!}}{k!(2m+1-k)!}.
\end{equation}
In summary, we have
\begin{enumerate}
\item[$(1)$]
$[{\hat{0}}_1,\hat{1}]=[{\hat{0}}_2,\hat{1}]=Q={{{\boxplus^\alpha}{(B_{2m+1})}}}.$
\item[$(2)$] The number of elements of rank $k+1$ in $P$ is
the same as the number of elements of rank $k$ in intervals
$[{\hat{0}}_1,\hat{1}]$ and $[{\hat{0}}_2,\hat{1}]$, $1\leq k\leq
2m+1$. \item[$(3)$]  $P$ has only two atoms
${\hat{0}}_1,{\hat{0}}_2.$
\end{enumerate}
Statements $(1), (2), (3)$ imply that
$P=P^{'}=\Sigma^{*}({{{\boxplus^\alpha}{(B_{2m+1})}}})$, as desired.

\end{proof}

\

\begin{theorem}\label{3even}
Let $P$ be an Eulerian Sheffer poset of odd rank $n=2m+1\geq 5$ with
 $B(3)=6$ and $D(3)=4$.
 Then $P={{\boxplus^\alpha}{(\Sigma^{*}(B_{2m}))}}$.
\end{theorem}
 \begin{proof}
We obtain the poset $Q$ by adding $\hat{0}$ and $\hat{1}$ to a
connected component of $P-\{\hat{0},~\hat{1}\}$. It is easy to see
that $Q$ is an Eulerian Sheffer poset  and also $P$ and $Q$ have the
same factorial functions and coatom functions up to rank $2m$. That
is $B_Q(k)=B_P(k)$ and $D_Q(k)=D_P(k)$ for $1\leq k\leq 2m$. By
$B(k)$, $D(k)$, $C(k)$, $A(k)$ we denote the factorial functions and
the coatom functions and atom functions of  $Q.$

By Theorem ~\ref{even}, $Q$ has the binomial factorial functions
$B(k)=k!$ for $1\leq k\leq 2m$. We have $C(2m+1)\geq A(2m)=2m$.
Since every interval of rank $2$ in the $Q$ is isomorphic to $B_2$,
 $Q$ has at least two coatoms. For every coatom such as $a_i$ in
$Q$, Theorem ~\ref{2even} imply that
$[\hat{0},a_i]=\Sigma^{*}({\boxplus^\alpha}(B_{2m-1}))$, by
considering the factorial functions we conclude that $\alpha =1$ as
well as $[\hat{0},a_i]=\Sigma^{*}(B_{2m-1})$. Since $Q$ is obtained
by adding $\hat{0}, \hat{1}$ to a connected component of
$P-\{\hat{0},\hat{1}\}$, we conclude that there are  at least two
particular coatoms $a_1, a_2$ such that there is an element  $c \in
[\hat{0},a_1],~[\hat{0},a_2]$ where $c\neq \hat{0}$. By considering
the interval $[c, \hat{1}]$ factorial functions, Theorems
~\ref{even} and~\ref{odd} imply that there is a positive integer $k$
such that $[c, \hat{1}]=B_k.$ Therefore, there is an element $b$ of
rank $k-2$ in  $[c, \hat{1}]$ such that $b=a_1\wedge a_2$, $b$ is
also an element of rank $2m-2$ in $Q$. The interval $[\hat{0}, b]$
is subinterval of $[\hat{0}, a_1]$ , so $[\hat{0},
b]=\Sigma^{*}(B_{2m-2}).$ We conclude $[\hat{0}, b]$ only has two
atoms say $x_1, x_2$. Since
$[\hat{0},a_1]=[\hat{0},a_2]=\Sigma^{*}(B_{2m-1})$, so the intervals
$[\hat{0},a_1]$ and $[\hat{0},a_2]$ only have two atoms $x_1$ and
$x_2$.

 Define a graph $G_Q$ as follows; vertices of $G_Q$ are coatoms of poset $Q$ and two vertices
(coatoms) $a_i$ and $a_j$ adjacent in $G_Q$ if and only if there is
an element $d\neq \hat{0}$ such that $d \in [\hat{0},a_i],
[\hat{0},a_j].$ Since $Q$ is obtained by adding $\hat{0}, \hat{1}$
to a connected component of $P-\{\hat{0},\hat{1}\}$,  $G_Q$ is a
connected graph.
 Thus, every coatom of rank $2m$ in $Q$ is just
above two atoms $x_1, x_2$ in $Q$. Hence the number of elements of
rank $1$ in poset $Q$ is $2$, and by Eq.(\ref{q0}),
\begin{equation}\label{eq1}
\frac{C(2m+1)D(2m)}{B(2m)}=2.
\end{equation}
Thus, $C(2m+1)=2m$ and also $Q$ has the same factorial function as
$\Sigma^{*}(B_{2m})$. By the same argument as Theorem~\ref{2even},
we conclude that $Q=\Sigma^{*}(B_{2m})$. So
$P={{\boxplus^\alpha}{(\Sigma^{*}(B_{2m}))}}$ for some positive
integer $\alpha$, as desired.
\end{proof}

\subsection{\bf{Characterization of the structure and factorial functions of  Eulerian
Sheffer posets of rank $n\geq 5$ with  $B(3)=4$ }}\label{b4}

\ In this section, we characterize Eulerian Sheffer posets of rank
$n\geq 5$ with  $B(3)=4$. Let $P$ be an Eulerian Sheffer poset of
rank $n\geq 5$ with $B(3)=4$. It can be seen that the poset $P$
satisfies one of the following cases:
\begin{enumerate}
\item $P$ has the following binomial factorial functions $B(k)=
2^{k-1}$, where $1\leq k\leq n-1;$ \item $n$ is even and there is a
positive integer $\alpha >1$ such that poset $P$ has the binomial
factorial functions $B(k)= 2^{k-1}$ for $1\leq k\leq n-2$ and
$B(n-1)={\alpha}.2^{n-2}$.
\end{enumerate}
\

As a consequence of Theorems $3.11$ and $3.12$ in~\cite{MR1}, we
characterize posets in the case $(i)$. \

  Theorem~\ref{6evem} deals with the case $(ii)$. It shows that if the Eulerian Sheffer posets $P$ of rank $n=2m+2$
has the binomial factorial functions $B(k)= 2^{k-1}$ for $1\leq
k\leq 2m$ and $B(2m+1)={\alpha}.2^{2m}$, where $\alpha> 1$ is an
integer, then $P=\Sigma^{*}{{{\boxplus^\alpha}{(T_{2m+1})}}}.$ See
Figure~\ref{fig:5}.

\

Given two ranked posets $P$ and $Q$, define the {\emph rank product}
$P*Q$ by
$$ P * Q
  =
   \{(x,z) \in P \times Q \:\: : \:\: \rho_{P}(x) = \rho_{Q}(z) \} .  $$
Define the order relation by $(x,y) \leq_{P * Q} (z,w)$ if $x
\leq_{P} z$ and $y \leq_{Q} w$. The rank product is also known as
the Segre product; see~\cite{Bjorner_Welker}.

\

\begin{theorem}\label{4even}[Consequence of Theorem 3.11~\cite{MR1}]
Let $P$ be an Eulerian Sheffer poset of rank $n\geq 4$ with the
binomial factorial functions $B(k)= 2^{k-1}$ for $1\leq k\leq n-1$.
Then its coatom function $C(k)$ and $P$ satisfy the following
conditions:
\begin{enumerate}
\item[$(i)$] $C(3)\geq 2$, and a length $3$ Sheffer interval is
isomorphic to a poset of the form $P_{q_1,\dots, q_r}$, as described
before. \item[$(ii)$] $C(2k)=2$, for ${\lfloor
\frac{n}{2}\rfloor}\geq k\geq 2$ and the two coatoms in a length
$2k$ Sheffer interval cover exactly the same element of rank $2k-2$.
\item[$(iii)$]$C(2k+1)=h$ is an even positive integer for
${\lfloor \frac{n-1}{2}\rfloor} \geq k\geq 2$. Moreover, the set of
$h$ coatoms in a Sheffer interval of length $2k+1$  partitions into
$\frac{h}{2}$ pairs,
$\{c_1,d_1\},\{c_2,d_2\},\dots,\{c_{\frac{h}{2}},d_{\frac{h}{2}}\}$,
such that $c_i$ and $d_i$ cover the same two elements of rank $2k-1$
\end{enumerate}
\end{theorem}

\

\begin{theorem}\label{5even}[Consequence of Theorem 3.12~\cite{MR1}]
Let $P$ be an Eulerian Sheffer poset of rank $n>4$ with the binomial
factorial functions $B(k)= 2^{k-1}$, $1\leq k\leq n-1$ and the
coatom functions $C(k)$, $1\leq k\leq n$. Then a Sheffer
$k$-interval $[\hat{0},y]$ of $P$ factors in the rank product as
$[\hat{0},y]\cong (T_{k-2}\cup \{\hat{0},\hat{-1}\})* Q$, where
$T_{k-2}\cup \{\hat{0},\hat{-1}\}$ denotes the butterfly interval of
rank $k-2$ with two new minimal elements attached in order, and $Q$
denotes a poset of rank $k$ such that
\begin{enumerate}
\item[$(i)$]each element of rank $2$ through $k-1$ in $Q$ is
covered by exactly one element, \item[$(ii)$] each element of rank
$1$ in $Q$ is covered by exactly two elements, \item[$(iii)$] each
element of even rank $4$ through $2\lfloor\frac{k}{2}\rfloor$ in $Q$
covers exactly one element, \item[$(iv)$] each element of odd rank
$r$ from $5$ through $2\lfloor\frac{k}{2}\rfloor +1$ in $Q$ covers
exactly $\frac{C(r)}{2}$ elements, and \item[$(v)$]each $3$-interval
$[\hat{0},x]$ in $Q$ is isomorphic to a poset of the form
$P_{q_1,\dots,q_r}$, where $q_1+\dots+q_r=C(3).$
\end{enumerate}
\end{theorem}

\

\begin{figure}[htp4]
\begin{center}
\includegraphics[height=1.5in]{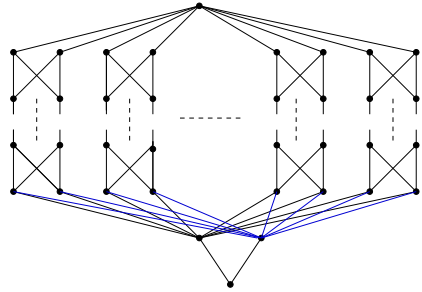}
\caption{$P=\Sigma^{*}({{{\boxplus^\alpha}{(T_{2m+1})}}})$ }
\label{fig:5}
\end{center}
\end{figure}

In the following theorem we study the only remaining case $(ii)$

 \

\begin{theorem}\label{6evem}
Let $P$ be an Eulerian Sheffer poset of even rank $n=2 m+2> 4$ with
the binomial factorial functions $B(k)= 2^{k-1}$ for $1\leq k\leq
2m$, and $B(2m+1)={\alpha}\cdot2^{2m}$, where $\alpha > 1$ is a
positive integer. Then
$P=\Sigma^{*}({{{\boxplus^\alpha}{(T_{2m+1})}}}).$
\end{theorem}
\begin{proof}
 Let $D(k)$, $1\leq k\leq 2m+2$, and also  $B(k)$, $1\leq k\leq 2m+1$,
be the Sheffer and binomial factorial functions of poset $P$,
respectively. The Euler-Poincar\'e relation for interval of size
$2m+2$ states as follows,

\begin{equation}\label{l1}
1+\sum_{k=1}^{2m+2}{(-1)^{k}\cdot\frac{D(2m+2)}{D(k)B(2m+2-k)}}=0.
\end{equation}

 The above Euler-Poincar\'e relation for the interval of  even rank  $2m+2$ can also  be stated as
follows,
\begin{equation}\label{l2}
\frac{2}{D(2m+2)}+\sum_{k=1}^{2m+1}{\frac{(-1)^{k}}{D(k)B(2m+2-k)}}=0.
\end{equation}
By expanding the left side of Eq.(\ref{l2}), we have:
\begin{equation}\label{l3}
\frac{(-1)}{{\alpha}\cdot{2^{2m}}}+\sum_{k=2}^{2m+2}{\frac{(-1)^{k}}{D(k)\cdot2^{2m+2-k-1}}}=0.
\end{equation}
Here, Eq.(\ref{l2}) for Sheffer $2m$-intervals can be stated as
follows,

\begin{equation}\label{l4}
\sum_{k=1}^{2m}{\frac{(-1)^{k}}{D(k)\cdot2^{2m-1-k}}}=0.
\end{equation}
Thus,
\begin{equation}\label{l5}
\frac{1}{2^{2m}}=\sum_{k=2}^{2m}{\frac{(-1)^k}{D(k)\cdot2^{2m+1-k}}}.
\end{equation}
It follows from Eq.(\ref{l3}) and Eq.(\ref{l5}) that

\begin{equation}\label{l6}
\frac{-1}{{\alpha}\cdot2^{2m}}+\frac{1}{2^{2m}}+\frac{-1}{D(2m+1)}+\frac{2}{D(2m+2)}=0.
\end{equation}
Let $k$ be the number of atoms in  a Sheffer interval of size
${2m+1}$ and $c=C(2m+2)$, so $D(2m+1)=k\cdot2^{2m-1}$ and
$D(2m+2)=ck\cdot 2^{2m-1}$. Therefore
\begin{equation}\label{l8}
\frac{1}{2^{2m}}-\frac{1}{{\alpha}\cdot2^{2m}}=\frac{1}{k\cdot2^{2m-1}}-\frac{1}{{\frac{ck}{2}}\cdot2^{2m-1}}.
\end{equation}

Thus,
\begin{equation}\label{l9}
\frac{1}{2}-\frac{1}{2\alpha}=\frac{1}{k}-\frac{1}{(\frac{c}{2})k}.
\end{equation}
 Comparing coatom and atom functions of Sheffer and binomial intervals, we have $k\geq 2$ as well as $c\geq
2{\alpha}$. By Eq.(\ref{l9}), we conclude that $k=2$ and
$c=2\alpha$. So $D(2m+1)=2^{2m}$ and
$D(2m+2)={\alpha}\cdot2^{2m+1}$. Since $B(2m+1)={\alpha}\cdot2^{2m}$
and $B(2m)=2^{2m-1}$, the number of atoms in poset $P$ is
$\frac{D(2m+2)}{B(2m+1)}=2.$ Since the only Eulerian Sheffer
interval of rank $2$ is $B_2,$ every Sheffer $j$-interval has  two
atoms for $1\leq j\leq 2m+1$.

$D(k)=2B(k-1)=2^{k-1}$ for $2\leq k\leq 2m+1$ as well as
$D(2m+2)={\alpha}\cdot2^{2m+1}.$ Let ${\hat{0}}_1$, ${\hat{0}}_2$ be
atoms of
 $P$. By Theorem~\ref{odd}, both intervals
$[{\hat{0}}_1,\hat{1}]$ and $[{\hat{0}}_2,\hat{1}]$ are isomorphic
to the poset $Q={{{\boxplus^\alpha}{(T_{2m+1})}}}$. It follows from
Eq.(\ref{q0}) that the number of elements of rank $k-1$ in the
intervals $Q=[{\hat{0}}_1,\hat{1}]=[{\hat{0}}_2,\hat{1}]$ is the
same as the number of elements of rank $k$ in poset $P$ and it can
be computed as follows,
\begin{equation}\label{qa}
\frac{D(2m+2)}{D(k)B(2m+2-k)}=\frac{B(2m+1)}{B(k)B(2m+1-k)}.
\end{equation}

We  know that ${\hat{0}}_1$,${\hat{0}}_2$ are the only atoms in $P$,
so by the above fact we conclude that
$P=\Sigma^{*}Q=\Sigma^{*}({{{\boxplus^\alpha}{(T_{2m+1})}}}) $, as
desired.


\end{proof}

\section{Finite Eulerian triangular posets}\label{triangular}

As we discussed before, the larger class of posets to consider are
triangular posets. For definitions regarding triangular posets, see
Section~\ref{def}. A non-Eulerian example of  triangular poset is
the the face lattice of the $4$-dimensional regular polytope known
as the 24-cell. In the following theorem, we characterize the
Eulerian triangular posets of rank $n\geq 4$ such that $B(k,k+3)=6$
for $1\leq k\leq n-3.$

\

\begin{theorem}\label{triangle}
Let $P$ be an Eulerian triangular poset of rank $n\geq 4$ such that
for every $0\leq k\leq n-3,$ $B(k,k+3)=6$. Then $P$ can be
characterized as follows:
\begin{enumerate}
\item $n$ is odd, there is an integer $\alpha \geq 1$ such that
$P={{\boxplus^{\alpha}}{(B_{n})}}$. \item $n$ is even, then $P=B_n$.
\end{enumerate}
\end{theorem}
\begin{proof}
We proceed by induction on  the rank of poset $n$.
\begin{enumerate}
\item[$\bullet$] $n=4$.
 A triangular poset of rank $4$ is also a Sheffer poset.  Since $B(1,4)=6$, by  Lemma~\ref{thm4} we conclude that $P=B_4$.

\item[$\bullet$] $n=2m+1$. By induction hypothesis, every
interval of rank $k\leq 2m$ in $P$ is isomorphic to $B_k$. Hence $P$
is a Sheffer poset and Theorem~\ref{sheffer6} implies that
$P={{\boxplus^{\alpha}}{(B_{n})}}$, where $\alpha \geq 1$ is a
positive integer.
 \

\item[$\bullet$] $n=2m+2$. Let $r$ and $t$ be the number of
elements of rank $1$ and $2m+1$ in $P$.  By induction hypothesis,
there are positive integers $k_t$ and $k_r$ such that
$B(1,2m+2)=k_t(2m+1)!$ and $B(0,2m+1)=k_r(2m+1)!.$ Therefore,
$B(0,2m+2)=tk_r(2m+1)!=rk_t(2m+1)!$ and also $B(n,n+k)=k!$, where
$1\leq k\leq 2m+1-n$ and $n\geq 1$. The Euler-Poincar\'e relation
for interval of size $2m+2$ state as follows,
\begin{equation}\label{t1}
1+\sum_{k=1}^{2m+2}{\frac{(-1)^k B(0,2m+2)}{B(0,k)B(k,2m+2)}}=0.
\end{equation}
By substituting the values in Eq.(\ref{t1}), we have
\begin{equation}\label{t2}
1+tk_r\left(\sum_{k=2}^{2m}{\frac{{(-1)}^{k}(2m+1)!}{k!(2m+2-k)!}}\right)+
\frac{-tk_r(2m+1)!}{k_t(2m+1)!} + \frac{-tk_r(2m+1)!}{k_r(2m+1)!} +1
=0.
\end{equation}
Eq.(\ref{t2}) lead us to
\begin{equation}\label{t3}
2-t\left(\frac{k_r}{k_t}+\frac{k_r}{k_r}\right)+tk_r\left(\sum_{k=2}^{2m}\left({\frac{(-1)^k(2m+1)!}{k!(2m+2-k)!}}\right)\right)=0,
\end{equation}
so,

\begin{equation}\label{t5}
2=t\left(\frac{k_r}{k_t}+\frac{k_r}{k_r}+{\frac{-(k_r)(4m+2)}{2m+2}}\right).
\end{equation}

Without loss of generality, let us assume that $k_r \geq k_t \geq
1$. Therefore,
\begin{equation}\label{t6}
2=t\left(\frac{k_r}{k_t}+\frac{k_r}{k_r}+{\frac{-(k_r)(4m+2)}{2m+2}}\right)\leq
t\left({k_r + 1}-\left(\frac{4m+2}{2m+2}\right)k_r\right) \leq
t\left(1-{\frac{2m}{2m+2}}k_r\right).
\end{equation}

 The right-hand side of the above equation is positive only
if $k_r=1$. So $k_r=1$ and since $k_r \geq k_t \geq 1$, we conclude
that $k_t=1$. Therefore, $2=t{\frac{2}{2m+2}}$ and so $t=2m+2$.
Similarly, we conclude  that $r=2m+2$. Thus, $P$ has the same
factorial function as $B_{2m+2}$ and by Proposition ~\ref{Bn}, this
poset is isomorphic to $B_{2m+2}$, as desired.

\end{enumerate}
\end{proof}

\section{Conclusions and remarks}\label{remark}

An interesting research problem is to classify the factorial
functions of Eulerian triangular posets. It is also interesting to
classify Eulerian trianguler posets with specific factorial
functions on their smaller intervals. In Theorem~\ref{triangle}, we
characterize the Eulerian triangular posets of rank $n\geq 4$ such
that $B(k,k+3)=6$, for $1\leq k\leq n-3.$

\

Readers can find the following result of Stanley. A graded  poset
$P$ is a boolean lattice if every $3$-interval is a boolean lattice
and for every $[x,y]$ of rank of at least $4$ the open interval
$(x,y)$ is connected (See \cite{G}, Lemma~8). Using Stanley's
result, it might be possible to obtain different proofs for Theorems
~\ref{even},~\ref{odd},~\ref{sheffer6} and~\ref{triangle}.

\

This research is motivated by the above result of Stanley. We
characterize Eulerian binomial and Sheffer posets by considering the
factorial functions of $3$-intervals. The project of studying
Eulerian Sheffer posets is almost complete. Only the following cases
remain to be studied:

\

\begin{itemize}
\item[$\bullet$]  Finite Eulerian Sheffer posets of odd rank with  $B(3)=6, D(3)=8.$ In
this case we ask the following question:
 Let $P$ be a Eulerian Sheffer poset of odd rank $n=2m+1\geq 5$ with  $B(3)=6, D(3)=8$ . Is there a positive integer $k$ such that  $P$ has the same
factorial function as the poset $Q_k={{\boxplus^k}{(C_{2m+1})}}$?

\item [$\bullet$] Finite Eulerian Sheffer posets of rank $5$ with $B(3)=6,
D(3)=10.$ We conjecture that there is no poset with these
conditions.
\end{itemize}



\section*{Acknowledgements}

The author would like to thank Richard Stanley for suggesting this
research problem. I also would like to thank Richard Ehrenborg,
Margaret A. Readdy and Richard Stanley for helpful discussions and
comments and thank Craig Desjardins for reading the draft of this
paper.

\end{document}